\documentclass[12pt,centertags,oneside]{amsart}
\RequirePackage[utf8]{inputenc}
\usepackage{amstext,amsthm,amscd,typearea}
\usepackage{amsmath}
\usepackage{amssymb}
\usepackage{a4wide}
\usepackage{enumitem}
\usepackage[mathscr]{eucal}
\usepackage{mathrsfs}
\usepackage{comment}
\usepackage{pdfsync}
\usepackage{pdflscape}
\usepackage{appendix}
\usepackage{multicol}
\usepackage{tikz}
\usepackage{color}
\usepackage{charter}
\usepackage{color}
\usepackage{xcolor}

\usepackage[colorlinks=true,  
            linkcolor=cyan,   
            citecolor=blue,   
            urlcolor=cyan       
            ]{hyperref}

\usepackage[a4paper,width=16.2cm,top=3cm,bottom=3cm]{geometry}

\usetikzlibrary{calc}

\newtheorem{theorem}{Theorem}[section]
\newtheorem{definition}[theorem]{Definition}
\newtheorem{proposition}[theorem]{Proposition}
\newtheorem{corollary}[theorem]{Corollary}
\newtheorem{lemma}[theorem]{Lemma}
\newtheorem{remark}[theorem]{Remark}

\newcommand{\volume}{{\rm vol}}
\newcommand{\lov}{{\rm lov}}
\newcommand{\supp}{{\rm supp}}
\newcommand{\dist}{\mathop{\mathrm{dist}}\nolimits}

\newcommand{\C}{\mathbb{C}}

\newcommand{\J}{\mathbb{J}}

\newcommand{\T}{\mathbb{T}}
\newcommand{\R}{\mathbb{R}}
\renewcommand\P{\mathbb{P}}

\title[Green currents and large entropy measures]{Convergence towards Green currents and supports of measures of large entropy on K\"ahler manifolds}

\author[F.~Bianchi]{Fabrizio Bianchi}
\address{Dipartimento di Matematica, Università di Pisa, Largo Bruno Pontecorvo 5, 56127 Pisa, Italy}
  \email{fabrizio.bianchi$@$unipi.it}

\author[S.~Boymurodov]{Sobir Boymurodov}
   \address{V.I.Romanovskiy Institute of Mathematics, Uzbekistan Academy of Sciences, 9, Universitet str., 100174, Tashkent, Uzbekistan}
\email{sboymurodov.research$@$gmail.com}
\author[K.~Rakhimov]{Karim Rakhimov}

\address{V.I.Romanovskiy Institute of Mathematics, Uzbekistan Academy of Sciences, 9, Universitet str., 100174, Tashkent, Uzbekistan}
\email{karimjon1705$@$gmail.com }

\begin{document}

\begin{abstract}
Let $f$ be a holomorphic automorphism of a compact K\"ahler manifold
$X$ of dimension $k\geq 2$,
with simple action on cohomology. 
For such maps, there exist an integer $0<p<k$ and Green currents $\T^+$ and $\T^-$ of bidegrees $(p,p)$ and $(k-p,k-p)$, respectively.
 We prove a new quantitative convergence theorem towards the Green  currents,
 extending previous exponential estimates by using a more flexible
norm defined through decompositions into positive currents that are not necessarily closed.
As an application, we   show  that every ergodic measure with sufficiently large entropy is supported on
$\supp \,\T^+\cap\supp \,\T^-$.
\end{abstract} 

\keywords{Automorphisms of K\"ahler manifolds, Julia sets, Green currents, entropy}

\noindent

\maketitle
{\textbf{AMS classification:}}
32H50, 32U40, 37A35, 37F80

\section{Introduction}
The study of the dynamics of holomorphic automorphisms
on compact K\"ahler manifolds
is a central topic in complex dynamics. We refer
to 
\cite{BK09,FilipTosatti24,McM06,McM07,Oguiso09,OT15} for interesting examples,
 to
\cite{ Dinh05JGA, DS10JAG,DS05, DTD12}
for the basic properties of these maps,
and to \cite{Cantat01, Cantat14, DH22, DS4, FT21,Oguiso14}
for more recent developments.
Let $(X,\omega)$ be a compact K\"ahler manifold of dimension $k$, and 
$f$
a holomorphic automorphism of $X$. For 
every 
integer $0\leq q\leq k$, the $q$-th dynamical degree of $f$ is defined as the spectral radius of the pull-back operator 
$f^*: H^{q,q}(X,\mathbb{C}) \to H^{q,q}(X,\mathbb{C})$.
We denote it by $d_q(f)$, or simply $d_q$ when the context is clear. It follows that $d_0 = d_k = 1$, and for each $n \geq 1$ we have
$d_q(f^n) = (d_q)^n$,
where $f^n = f \circ \cdots \circ f$ ($n$ times) is the $n$-th iterate of $f$.  By Poincar\'e duality, 
we have 
$d_q(f)= d_{k-q}(f^{-1})$ for every $q$. 
It follows from
the fundamental
Gromov-Khovanskii-Teissier inequalities
 \cite{Gromov90,Khovanskii,Teissier} that the sequence
${q\mapsto\log{d_q}}$ is concave. So, there are integers $0\le p\le p'\le k$ such that
$1=d_0<\dots<d_p=\dots=d_{p'}>\dots>d_k=1$.
 Throughout this paper
 we shall assume that $f$ has
 \emph{simple action on cohomology}. 
 This means that  $p=p'$ and 
 that the action of $f^*$ on $H^{p,p}(X,\mathbb{R})$ 
 has a unique eigenvalue of maximal modulus, denoted by $d_p$, which we call the \emph{main dynamical degree} of $f$.  We
define $\delta_f$ to be the maximum between  
$\max\{d_{p-1},d_{p+1}\}$
 and the absolute values of
  all the eigenvalues different from $d_p$
  for the action of $f^*$
 on $H^{p,p}(X,\mathbb{R})$. Hence, simple 
 action means 
 $\delta_f<d_p$ {and $d_p$ is 
 a simple eigenvalue}.

Automorphisms with simple action have been especially studied  in \cite{BD,DS10,LZ26,VW26} 
as this condition can be seen as a cohomological version of hyperbolicity.
  By \cite{DS05},
  there exist  positive closed currents $\T^+$ and $\T^-$, 
  unique up to multiplication by positive constants,
  of bidegrees
$(p,p)$ and $(k-p,k-p)$,
   that are invariant under $d_p^{-1}f^*$ and $d_p^{-1} f_*$, respectively, that is,
 that satisfy
$f^*(\T^+) = d_p \T^+$
and
$f_*(\T^-) = d_p \T^-$.
The currents $\T^+$ and $\T^-$ are called the \emph{Green currents} of $f$ and $f^{-1}$, respectively. They can be obtained,
up to positive multiplicative constants,
as the limit currents of
the sequences 
$d_p^{-n}(f^n)^*(\omega^{ p})$ and $d_p^{-n}(f^n)_*(\omega^{k-p})$, respectively. 
The Green currents $\T^\pm$ have H\"older continuous super-potentials, and their intersection
$\T^+ \wedge \T^- := \mu$
is well-defined and is referred to as the \emph{Green measure}, which is totally invariant under both $f^*$ and $f_*$. 
We normalize $\T^+$ and $\T^-$ so that
$\T^+\wedge \T^-$ has mass $1$.
The Green measure
$\mu$
is the unique 
probability measure of  maximal entropy $\log d_p$ and satisfies strong ergodic and statistical properties, see for instance
\cite{BD,DS10,DS10JAG}.
By
definition, 
we have $\supp\, \mu \subseteq 
{\supp \,\T^+\cap\supp \,\T^-}$.
\subsection{Main Results}
 We introduce 
 a new
 dynamically adapted norm $\|\cdot\|^{\bigstar}$ and extend 
to the 
 weaker
$\|\cdot\|^{\bigstar}$-topology 
 the super-potential theory developed by Dinh-Sibony with respect to the stronger $\|\cdot\|_*$ norm. 
Let  $\mathscr{D}_p(X)$ denote
  the  real space generated by positive closed $(p,p)$-currents on $X$.
Our first main result is the following quantitative equidistribution
theorem for currents. 
\begin{theorem}\label{thm:new-speed}
Let $f$ be a holomorphic automorphism of $(X,\omega)$ with simple action on
cohomology, let $d_p$ be its main dynamical degree, and 
 fix
$\delta_f<\delta<d_p$. Take $S\in\mathscr{D}_p(X)$ with
$\|S\|_*<\infty$, and let $c_S$ be the constant such that
$d_p^{-n}(f^n)^*S$ converges to $c_S\T^+$. Let $R$ be a
$(k-p,k-p)$-current such that
$\|R\|<\infty$, $\|dd^cR\|^{\bigstar}<\infty$, and
$\|\mathscr U_{dd^cR}\|_{\gamma,*}<\infty$ for some $0<\gamma\leq1$.
Then there exists a constant $B>0$, independent of $R$, $S$, and $n$, such
that, for every $n\in\mathbb N$,
\[
\left|\left\langle d_p^{-n}(f^n)^*S-c_S\T^+,R\right\rangle\right|
\leq B\left(\frac{\delta}{d_p}\right)^n\|S\|_*
\left(\|dd^cR\|^{\bigstar}
+\|\mathscr U_{dd^cR}\|_{\gamma,*}+\|R\|\right).
\]
\end{theorem}
A result similar  to Theorem~\ref{thm:new-speed} is established in \cite{DS10}
(see also \cite[Corollary~3.3]{BD}) with $\|dd^cR\|_*$ instead of
$\|dd^cR\|^{\bigstar}$.
 The
definition of the $*$-norm is based on
decompositions into positive closed currents.
We observe also
that recently Vergamini \cite{Vergamini25}
gave a more general version of the above convergence,
where only the 
log-H\"older-continuity 
 of the  super-potentials $\mathscr U_{dd^c R}$
 (with respect to the $*$-topology)
is required. 
 The novelty of
Theorem~\ref{thm:new-speed} is that $dd^cR$ is controlled in the more
flexible $\bigstar$-norm, whose definition allows decompositions into positive
currents that are not required to be closed.

\smallskip

As mentioned above, the Green measure $\mu$ satisfies $\supp\,\mu\subseteq 
 \supp \,\T^+\cap\supp \,\T^-$.
The  next goal of this paper is to
prove the same inclusion for every ergodic measure of sufficiently large entropy.
We refer for instance 
to \cite[{Section}~7.3]{Cantat14}
 for examples where $ \supp \,\T^+\cap\supp \,\T^-\neq X$.
  
\begin{theorem}\label{thm: dimX=k, p}
    Let  $(X,\omega)$ be   a compact K\"ahler manifold 
    and 
    let $f:X\to X$ be a holomorphic automorphism with simple action on cohomology.
    Then every ergodic measure $\nu$ with
    entropy $h_{\nu}(f)>\log\delta_f$ is supported on $\supp \T^+\cap \supp \T^-$.  
\end{theorem}
We observe that a weaker localization result
follows from normality
arguments. Indeed, normality on the Fatou set
and Ruelle's inequality imply that every ergodic
invariant probability measure of positive entropy
is supported in the complement of the Fatou set,
see  \cite[\S 6]{DS05v1} and
\cite[\S 10]{CantatDupont20}.
In particular, for automorphisms of compact K\"ahler surfaces
with positive entropy, this complement coincides
with $\supp T^+\cup\supp T^-$, where $T^+$ and
$T^-$ denote the Green $(1,1)$-currents;
see \cite[Theorem~7.6]{Cantat14}.
This localization in the union does not imply
our conclusion, which concerns the intersection.
Moreover, in higher dimensions the main Green
currents $\T^+$ and $\T^-$ appearing in
Theorem~\ref{thm: dimX=k, p} have bidegrees
$(p,p)$ and $(k-p,k-p)$, and need not be
$(1,1)$-currents.
In Section~\ref{section:example}, we show
that 
$\supp\T^+\cup\supp\T^-$
 (and so also $\supp\T^+\cap\supp\T^-$)
can be
strictly smaller than the complement of the
Fatou set.
 In dimension two, stronger results
than Theorem \ref{thm: dimX=k, p}
 are already
known; 
see
\cite[Theorem~6.12]{Cantat14}.
The main interest of our theorem therefore
lies in higher dimensions.

\smallskip

Parallel results
to Theorem \ref{thm: dimX=k, p} 
have been obtained in the context of holomorphic endomorphisms of 
complex projective spaces
\cite{deT, Dinh-attracting-current}
and more recently
for polynomial-like maps on $\mathbb{C}^k$
\cite{BBR}. 
Observe that, in these cases,
the maximal dynamical degree is
the topological degree
$d_k$,
hence the problem is simpler as only expanding directions are present. 
For H\'enon-like mappings in $\mathbb{C}^2$, and in particular for Hénon maps in $\mathbb C^2$, 
related results were obtained in \cite{Dujardin2004,Dujardin2020}.

A key point of the proof
in the case of endomorphisms of projective spaces
in \cite{deT,Dinh-attracting-current}
is that the $(1,1)$-Green current $T$ 
and the Fubini--Study form lie in the same cohomology class. This  method cannot be applied in our setting
as the
Green currents $\T^{\pm}$
are not necessarily cohomologous to
{powers of}
K\"ahler forms
(it can, on the other hand, be applied to
the so-called H\'enon-Sibony maps, or regular automorphisms of $\mathbb C^k$ \cite{Dinh05, Sibony}, see
\cite{Boy26}).
As cohomology tools are also not available in the (non-algebraic) setting of polynomial-like maps,
the proof in that case given 
in \cite{BBR}
relied on a different idea, based on the
\emph{exponential
speed of convergence}
of preimages of points  towards
the Green measure \cite{DS03, DS5}.

\smallskip

The proof of Theorem \ref{thm: dimX=k, p} will be closer in spirit to that of \cite{BBR}. 
However, instead of considering preimages of points
(which -- since $f$ is an automorphism
-- cannot 
equidistribute),
we 
will use the exponential speed of convergence from Theorem~\ref{thm:new-speed}.
In our setting, since we will need to work on a given set (the complement of {$\supp\T^\pm$}),
we will be naturally forced to work with \emph{non-closed}
currents. 
The quantitative convergence result of Theorem~\ref{thm:new-speed}, formulated in the more flexible $\bigstar$-norm, is precisely suited to this situation and provides the main tool for the proof of
 Theorem \ref{thm: dimX=k, p}.
We observe that this difficulty
is specific (and intrinsic)
 when the dimension  $k$
of $X$ is at least $3$. In the special case 
$k=2$,
we
would not need to introduce a new norm
(as, for dimensional reasons, the currents we use would be automatically closed), 
and we could apply the convergence speed of \cite{DS10}, see Remark \ref{r:dim2}  
 for more details and \cite{Dujardin2004,Dujardin2020}
 for related results.
It would also be interesting to understand whether analogous results hold for
non-invertible maps, holomorphic correspondences, or in 
local settings such as
H\'enon-like maps. We briefly discuss these questions in Section~\ref{s:final}.

\smallskip

{\bf Outline of the paper.}
In Section~\ref{section:2}, we recall
the 
notions concerning currents and  
super-potentials.   
 We also introduce the
 dynamically adapted norm $\|\cdot\|^\bigstar$ for currents, extend the notion of super-potentials with respect to this norm, and prove a Skoda-type inequality in Section~\ref{section:3-the-big-star-norm}.
In Section \ref{s:action}
 we 
prove Theorem \ref{thm:new-speed}.
In Section~\ref{sec:decay_estimates}
we 
give the main dynamical estimates that
permit us
to control the action of $f^*$ outside $\supp \T^\pm$. 
The proof
 of Theorem \ref{thm: dimX=k, p} is then completed in Section \ref{s:proof-thm}. 
In Section~\ref{section:example}, we give
an example showing
that the
complement of the Fatou set can strictly
contain the union of the supports of the
main Green currents, and that the entropy
threshold in Theorem~\ref{thm: dimX=k, p}
is optimal. 
 We collect in Section \ref{s:final}
 some final remarks and related open questions.
In the Appendix, we give a complete
proof
of 
some
technical results which are needed in the proof
 of Theorem \ref{thm:new-speed}.
\smallskip

{\bf Acknowledgments.}
 The first author would like to thank the
National University of Uzbekistan and the 
V.I. Romanovskiy Institute of Mathematics, Uzbekistan Academy of Sciences
for the warm welcome and the excellent work conditions.
 The authors would also like to thank   Romain Dujardin and Marco Vergamini  
 for useful discussions and helpful comments.
This project has received funding from
 the French government through the Programme
 Investissement d'Avenir
(ANR QuaSiDy /ANR-21-CE40-0016,
ANR PADAWAN /ANR-21-CE40-0012-01,
ANR TIGerS/ANR-24-CE40-3604)
and from the MIUR Excellence Department Project awarded to the Department of Mathematics, University of Pisa, CUP I57G22000700001.
The first
author is affiliated to the GNSAGA group of INdAM.

\smallskip

{\bf AI declaration.}
The first
version of this paper appeared on arXiv in June 2026. All the mathematical arguments are due to the authors. 
The authors used OpenAI’s ChatGPT for editing assistance during the revision of the paper.

\section{Preliminaries}\label{section:2}
 In this section, we recall some basic properties of positive closed currents on a compact K\"ahler manifold, together with the notion of super-potentials.
\subsection{Currents on a K\"ahler manifold}
Let $(X,\omega)$ be a compact K\"ahler manifold 
with $\dim X=k$
and $\volume(X)=\int_{X}\omega^k=1$.
 For every 
  $0\le q\leq k$ denote
   by  $H^{q,q}(X,\mathbb{C})$ the Hodge cohomology group of bidegree $(q,q)$ of $X$. If $R$ is a  closed $(q,q)$-current, its 
   cohomology
   class in $H^{q,q}(X,\mathbb{C})$ is   denoted by $\{R\}$. If $R$ is a current  on $X$ and $\Phi$ is a test form of complementary bidegree,
    the pairing $\langle R,\Phi\rangle$ denotes the value of $R$ at $\Phi$. In particular, for a positive $(q,q)$-current  $R$ on $X$, define
    the mass of $R$ by
   $\Vert R\Vert:=\langle R, \omega^{k-q}\rangle$.
When $R$ is  positive and closed, its mass depends only on its cohomology class $\{R\}$ in $H^{q,q}(X,\R)$.  Let
$\mathscr{C}_q(X)$ denote the cone of positive closed $(q,q)$-currents on $X$. 
Let  $\mathscr{D}_q(X)$ denote
  the  real space generated by positive closed $(q,q)$-currents on $X$, and by $\mathscr{D}_q^0(X)$ the subspace of $\mathscr{D}_q(X)$ given by  exact currents, i.e., currents $R$ whose class $\{R\}=0$.  
Following \cite{DS10JAG},
define  the norm $\Vert R\Vert_*$  of $R\in \mathscr{D}_q(X)$
 as 
\[\Vert R \Vert_*= \inf(\Vert R^+\Vert+\Vert R^-\Vert),\]
 where the 
 infimum
 is taken over the 
 \emph{positive closed}
 $(q,q)$-currents $R^{\pm}$ with $R=R^+-R^-$, and is reached by compactness.
 We can 
 equip
   $\mathscr{D}_q(X)$ and $ \mathscr{D}_q^0(X)$   with  the topology in which $R_n\to R$ if  $R_n$ converges to $R$ in the sense of currents and $\displaystyle\sup_n\Vert R_n\Vert_*<\infty$. Given an  integer $\ell>0$, we define on $\mathscr{D}_q(X)$ the norm 
   $\Vert\cdot\Vert_{\mathcal{C}^{-\ell}}$ and  the distance $\dist_{\ell}$ by 
$$\Vert R\Vert_{\mathcal{C}^{-\ell}}=\sup_{\Vert \Phi\Vert_{\mathcal{C}^{\ell}}\leq 1}|\langle R,\Phi\rangle| \quad \text{ and }\quad  \dist_{\ell}(R, R'):=\Vert R-R'\Vert_{\mathcal{C}^{-\ell}},$$ where $\Phi$ is a smooth  test form of bidimension $(q, q)$ on $X$. Observe that we can also extend the notion of mass  to any current
{of order zero},
as the dual of the semi-norm $\mathcal C^0$ on forms.
\subsection{Super-potentials}\label{subsection:superpotentials}
 {We now recall the notion of the
($\alpha$-normalized) super-potential 
of a current
$R\in\mathscr{D}_q(X)$. 
{This is first defined
on smooth test} currents in
$\mathscr{D}_{k-q+1}^0(X)$    for  $0< q\leq k$,  see \cite{DS10JAG,DinhSibony10,DNV18,KV19}
for more details.
 Fix
 $S \in \mathscr{D}_{k-q+1}^0(X)$. Since $\{S\} = 0$,
  there exists a
  (not necessarily unique) current $U_S$ of bidimension $(q,q)$ such that 
$S = dd^c U_S.$
We refer to $U_S$ as a \emph{potential} of $S$.
Consider a fixed family of  real smooth closed
$({q,q})$-forms 
$\alpha_{ q}=(\alpha_{ q,1},\dots \alpha_{ q,h(q)})$,
 where $ h({ q}):=\dim H^{ {q,q}}(X,\mathbb{R})$
 is such  that  the family of classes $\{\alpha_{ q}\}=(\{\alpha_{ q,1}\},\dots,\{\alpha_{ q,h(q)}\}) $ is a basis of $H^{ q,q}(X,\mathbb{R})$.
 We further choose smooth real closed $({ k-q,k-q)}$-forms
$\check{\alpha}_{{ q}} = (\check{\alpha}_{ q,1},\ldots,\check{\alpha}_{ q,h(q)})$ 
representing  a dual basis
of $\alpha_{ q}$ in  $H^{ k-q, k-q}(X,\mathbb{R})$ for the Poincar\'e duality.  
By adding  to 
the
potential $U_S$ a suitable combination 
of the ${\check{\alpha}_{q,j}}$'s
 we can assume 
 that $U_S$ is  $\alpha_{ q}$-\emph{normalized},
  i.e., 
  that $\langle U_S, { \alpha_{q,j}}\rangle=0$ for every $1\leq j\leq h({q})$.  When $S$ is smooth we can choose $U_S$ to be smooth and
  $\alpha_{{ q}}$-normalized.  The $\alpha_{{ q}}$-normalized 
  \emph{super-potential} $\mathscr{U}_R$ of $R\in \mathscr{D}_q(X)$
  is defined on  smooth 
exact
  forms $S\in \mathscr{D}_{k-q+1}^0$ by 
  the identity
  $\mathscr{U}_R(S):=\langle R, U_S\rangle$,
   where $U_S$ is an
   $\alpha_{ q}$-normalized potential of $S$.
   Observe that the expression above is independent of the potential $U_S$ of $S$, as soon as this is $\alpha_{ q}$-normalized.
{\subsection{A decomposition lemma} 
We conclude this section with the following lemma.
The proof is standard, but we give it for completeness.
\begin{lemma}\label{lem:phi_n^+}
Let $K$ be a compact set on $X$. Let $0 \leq s \leq k-1$ be an integer and
$\Phi$
a 
real smooth $(s,s)$-form on $X$ with  compact support outside $K$. 
There  exist two positive
smooth  $(s+1,s+1)$-forms
$\Phi_{1}$ and $\Phi_2$ such that
\[dd^c\Phi=\Phi_1-\Phi_2
\quad \text{ and }
\quad
\supp \Phi_{j}\cap K=\emptyset \quad \text{ for } j=1,2.\]
\end{lemma}
 
\begin{proof} 
Since $\Phi$ is compactly supported outside $K$, 
there exist open sets $W\Subset W_1$ such that 
$\supp\Phi\subset W\Subset W_1$ and $\overline{{W}_1}\cap K=\emptyset$.
The form $dd^c\Phi$ is smooth and supported in $W$. 
 Choose a smooth  cut-off function $\chi$
 on $X$ 
 {with $0\leq \chi\leq 1$ and}
 such that $\chi\equiv1$  on $W$ and $\supp\chi\Subset W_1$. 
Since $dd^c\Phi$ is smooth, there exists $M>0$ such that
$-M\omega^{s+1}\le dd^c\Phi\le M\omega^{s+1}$ on $W$.
  Fix a constant $C>M$. Consider the 
  smooth $(s+1,s+1)$-forms $\Phi_1:=C\chi\omega^{s+1}$
  and
  $\Phi_2:=C\chi\omega^{s+1}{-}dd^c\Phi$.
 It is clear that $\Phi_1$ is positive. We now show that $\Phi_2$ is also positive. Since it is compactly supported in $W_1$ it is enough to show that $\Phi_2$ is positive in $W_1$.
On 
$W$, we have
$\Phi_2 \ge (C - M)\omega^{s+1} > 0.$ On the remaining part
 $W_{1}\setminus W$, we have $\Phi_2 = C\chi\omega^{s+1}\geq 0$. Hence, $\Phi_1$ and $\Phi_2$ satisfy the  assertion.
\end{proof}
}

\section{The new \texorpdfstring{$\bigstar$}{}-norm and its properties}\label{section:3-the-big-star-norm}
In this section, we introduce the dynamically adapted { $\bigstar$-norm}, 
which plays a central role in our approach. We then study super-potentials with respect to 
{the corresponding}
topology  and establish the H\"older continuity properties that will be used throughout the article.

\subsection{The \texorpdfstring{$\bigstar$}{}-norm }\label{subsection:2.2}
For $0\leq q \leq k$, 
denote by $\mathscr{H}_q$ 
the set
\[
\mathscr{H}_{q}=\{\Phi\in\mathscr{C}_{q}(X) \colon \Phi \text{ is smooth},
\Vert\Phi\Vert\leq 1\}.\]

\begin{definition}
We 
define the 
$\bigstar$-norm on 
  $\mathscr{D}_q(X)$ as
\begin{equation}\label{def:bigstar_norm}
\Vert R\Vert^{\bigstar}:=\inf\sup_{\Phi\in\mathscr{H}_{k-q}}(\Vert R_1\wedge \Phi\Vert+\Vert R_2\wedge \Phi\Vert)
\end{equation} where
 the  infimum is taken over  all 
 \emph{positive} currents $R_1,R_2$ such that  $R=R_1-R_2$. 
 \end{definition}
 \begin{remark}
The main difference between the $\bigstar$-norm and the
$*$-norm
\cite{DS10JAG} is that the positive currents appearing
in the decomposition are not required to be closed.
This flexibility is particularly useful in the localization
arguments in Sections~\ref{sec:decay_estimates}
and~\ref{s:proof-thm}, especially in the construction
of the current $R$ in~(\ref{eq:SR}), where multiplying
positive currents by smooth cut-off functions naturally
produces non-closed positive currents.
\end{remark}

\begin{remark}
 The same expression
 \eqref{def:bigstar_norm}
 can be considered on the
space generated by arbitrary positive currents,
but it may take the value $+\infty$ even on
positive currents of finite mass, see Section \ref{s:norm-infinity}.
We therefore restrict ourselves to
$\mathscr D_q(X)$, where it is finite and
defines a norm, as shown
in Lemma \ref{l:bigstar-norm} below.
\end{remark}

\begin{lemma}\label{l:bigstar-norm}
The function  $R\mapsto \Vert R\Vert^{\bigstar}$ 
defines a norm on 
$\mathscr{D}_q(X)$.
\end{lemma}
\begin{proof}
If $R=0$, we can
 take $R_1=R_2=0$. Then $R=R_1-R_2$, and clearly 
$\Vert R\Vert^{\bigstar}=0$.
If $\Vert R\Vert^{\bigstar}=0$,
 for every   $n\in\mathbb{N}$ there exist positive currents $R_{1,n},R_{2,n}$ such that $R=R_{1,n}-R_{2,n}$ and $\sup_{\Phi\in\mathscr{H}_{k-q}}(\Vert R_{1,n}\wedge \Phi\Vert+\Vert R_{2,n}\wedge \Phi\Vert)\leq 1/n$. In particular, taking 
 $\Phi=\omega^{k-q}$,
  it follows that $\Vert R_{1,n}\Vert+\Vert R_{2,n}\Vert\leq 1/n$. Let $\Omega$ be a $(k-q,k-q)$-test form with $\|\Omega\|_{\infty}\le 1$.  Hence, 
  there exists a
constant $A > 0$ independent of $\Omega$ so that
$|\Omega|\leq  A\|\Omega\|_\infty \omega^{k-q}$
and 
$$
|\langle R,\Omega\rangle|
= |\langle R_{1,n},\Omega\rangle - \langle R_{2,n},\Omega\rangle|
\le A
\|\Omega\|_\infty \big( \|R_{1,n}\| + \|R_{2,n}\|\big).
$$
Letting $n\to\infty$, the right-hand side tends to zero, which implies
$\langle R,\Omega\rangle = 0$
for all test forms $\Omega$. Therefore, $R=0$.
It can also 
be checked directly that $\Vert\lambda R\Vert^{\bigstar}=|\lambda|\Vert R\Vert^{\bigstar}$ for every $\lambda\in{\mathbb{R}}$.
So it only
remains to prove the triangle inequality. 

Fix
$R, T\in\mathscr{D}_q(X)$. Write $R=R_1-R_2$ and $T=T_1-T_2$,
where $R_1,R_2,T_1,T_2$ are positive.
 Then we can decompose $R+T$ as $(R_1+T_1)-(R_2+T_2)$, and it follows that
\begin{align*}
\sup_{\Phi\in\mathscr{H}_{k-q}} & (\Vert(R_1+T_1)\wedge \Phi\Vert+\Vert(R_2+T_2)\wedge \Phi\Vert)\\
& \leq\sup_{\Phi\in\mathscr{H}_{k-q}}(\Vert R_1\wedge \Phi\Vert+\Vert R_2\wedge \Phi\Vert)+\sup_{\Phi\in\mathscr{H}_{k-q}}(\Vert T_1\wedge \Phi\Vert+\Vert T_2\wedge \Phi\Vert).
\end{align*}
As 
decompositions of $R$ and $T$ also give 
decompositions of $R+T$,
taking the infimum over all such
decompositions gives
$\Vert R+T\Vert^{\bigstar}\leq\Vert R\Vert^{\bigstar}+\Vert T\Vert^{\bigstar}$. 
This concludes the proof.\end{proof}

We say that a subset of $\mathscr{D}_q(X)$ is \emph{$\bigstar$-bounded}
if it is bounded 
with respect to
the  ${\bigstar}$-norm. 
We will
consider on $\mathscr{D}_q(X)$ and $\mathscr{D}_q^0(X)$ the following topology.
\begin{definition}
    We say that a sequence $\{R_n\}_{n\geq 1}$ in $\mathscr{D}_q(X)$ (or $\mathscr{D}_q^0(X)$) 
    $\bigstar$-converges to  a current $R$ if $R_n$ converges
     to $R$ in the sense of
     currents and if $\Vert R_n\Vert^{\bigstar}$ are bounded by a  constant independent of $n$.
\end{definition}
\begin{lemma}\label{lem:compare-two-norm}
We have
$\Vert R\Vert^{\bigstar} \lesssim \, \Vert R\Vert_*$ 
for every $R \in \mathscr{D}_q(X)$ and
$\Vert S\Vert^{\bigstar} \asymp \Vert S\Vert = \Vert S\Vert_*$
for every $S\in \mathscr{C}_q(X)$, where the implicit constants are independent of $R$ and $S$.
\end{lemma}
\begin{proof}
In the definition of
$\Vert R\Vert^{\bigstar}$,
 the decompositions of $R$ 
are taken over all
 positive currents $R_1, R_2 \ge 0$ such that $R = R_1 - R_2$. 
 In that of
 $\Vert R\Vert_*$, we require that $R_1, R_2 \ge 0$ are positive and closed.
    As
     the cone of positive closed currents is contained in the cone of positive currents,
      we obtain
    \begin{align*}
    \|R\|^{\bigstar} &= 
\inf_{\substack{ R_1, R_2 \ge 0}}
\sup_{\Phi \in \mathscr{H}_{k-q}} 
\big( \| R_1 \wedge \Phi \| + \| R_2 \wedge \Phi \| \big)\leq 
\inf_{\substack{ R_1, R_2 \ge 0\\ \rm closed}}
\sup_{\Phi \in \mathscr{H}_{k-q}} 
\big( \| R_1 \wedge \Phi \| + \| R_2 \wedge \Phi \| \big)\\
&\lesssim
\inf_{\substack{ R_1, R_2 \ge 0\\ \rm closed}}
\sup_{\Phi \in \mathscr{H}_{k-q}} 
\big( \Vert R_1\Vert\cdot\Vert\Phi\Vert+\Vert R_2\Vert\cdot\Vert\Phi\Vert\big)=\inf_{\substack{ R_1, R_2 \ge 0\\\rm closed}}(\Vert R_1\Vert+\Vert R_2\Vert)=\Vert R\Vert_*,
    \end{align*}
    where we use the fact
    that, since integrals can be computed cohomologically,
    we have $\|R_1\wedge \Phi \|\lesssim\|R_1\|\cdot \|\Phi\|$
    (and similarly for $R_2$)
for some
implicit constant independent of $R_1$ and $\Phi$.
This proves the first estimate in the statement.

\smallskip

For the second, fix 
 a positive closed $(q,q)$-current $S$. 
We can write $S = S - 0$.  Hence, in particular, we have
$\Vert S\Vert_*=\Vert S\Vert$.
As the mass of a positive closed current can be computed cohomologically, we have
  \begin{align*}
\Vert S\Vert^{\bigstar}\leq\sup_{\Phi\in\mathscr{H}_{k-q}}\Vert S\wedge\Phi\Vert\lesssim \sup_{\Phi\in\mathscr{H}_{k-q}}\Vert S\Vert\cdot\Vert\Phi\Vert=\Vert S\Vert,
  \end{align*}
 where again the implicit
  constant is independent of  $S$ and $\Phi$. 
For the opposite inequality,
by 
the definition of $\Vert S\Vert^{\bigstar}$,
 for every $\kappa>0$ there exist  positive currents $S_1, S_2{\geq}0$  
 such that  $S=S_1-S_2$ 
 and $\sup_{\Phi\in\mathscr{H}_{k-q}}(\Vert S_1\wedge\Phi\Vert+\Vert S_2\wedge\Phi\Vert)<\Vert S\Vert^{\bigstar}+\kappa$. 
 In particular,
 taking $\Phi = \omega^{k-q}$, we see that
 $\Vert S_1\Vert+\Vert S_2\Vert<\Vert S\Vert^{\bigstar}+\kappa$.
  Moreover, since $S$ is positive, we have
   $\Vert S_1\Vert+\Vert S_2\Vert=\Vert S+S_2\Vert+\Vert S_2\Vert=\Vert S\Vert+2\Vert S_2\Vert\geq\Vert S\Vert$.
   Thus, we get $\Vert S\Vert^{\bigstar}+ \kappa>\Vert S\Vert$ and letting $ \kappa\to 0$ yields $\Vert S\Vert^{\bigstar}\geq\Vert S\Vert$.
The assertion follows.
\end{proof}

\begin{corollary}\label{cor:density-resp-bigstar-norm}
Smooth forms are dense in $\mathscr{D}_{q}$ and in $\mathscr{D}_{q}^{0}$ for the $\bigstar$-topology.
\end{corollary} 
\begin{proof}
By  \cite[Theorem 2.4.4]{DS10JAG},
smooth forms are dense in  $\mathscr{D}_{q}$ and in $\mathscr{D}_{q}^{0}$ for the $*$-topology. By Lemma \ref{lem:compare-two-norm}, we have
$\|R\|^{\bigstar} \leq c \|R\|_*$ for every $R\in\mathscr{D}_{q}$, where the
 implicit constant is independent of $R$.
Therefore, any sequence $R_n$ that converges with respect to the $*$-topology also converges with respect to the $\bigstar$-topology. The assertion follows.
\end{proof}
}
 We refer to
Proposition \ref{prop:smooth-bigstar}
for a more quantitative version of Corollary \ref{cor:density-resp-bigstar-norm}.

  \subsection{Super-potentials {with respect to} the \texorpdfstring{$\bigstar$}{}-topology} \label{subsection:sup-potential-bigstar-topology} 
  Recall that we have introduced two topologies on
$\mathscr{D}_q(X)$, 
namely the $*$-topology and
the $\bigstar$-topology.
We denote by
\[\mathscr{D}_{q,*}^1(X) := \{R\in 
\mathscr{D}_q \colon \|R\|_*\leq 1\}
\quad \text{ and } \quad
\mathscr{D}_{q,\bigstar}^1(X)
:=\{R\in \mathscr{D}_q \colon \|R\|^\bigstar\leq 1\}
\]
the unit balls of
 $\mathscr{D}_q(X)$ with respect to 
the two 
norms.

\begin{definition}\label{def:Holder_continuity}
Let $\bullet$ be either $*$ or $\bigstar$.
 Take $R\in \mathscr{D}_q$
 and let $M, \gamma>0$ be
  positive constants.
\begin{enumerate}
\item We say that 
$R$
has a \emph{continuous super-potential}
with respect to the $\bullet$-topology
 if the function 
$\mathscr{U}_R$ extends continuously to $\mathscr{D}_{k-q+1}^0(X)$ 
with respect to
 the $\bullet$-topology.
\item
We say that 
the super-potential $\mathscr{U}_R$
of $R$
is 
$(M, \gamma)$-\emph{H\"older continuous}
with respect to
the $\bullet$-topology
if it is continuous
 with respect to the
 $\bullet$-topology
 and satisfies
$$
|\mathscr{U}_R(S)| \leq M \, \|S\|_{\mathcal{C}^{-2}}^{\gamma}
\quad \text{
for every 
}
S\in\mathscr{D}_{k-q+1}^0\cap\mathscr{D}_{k-q+1,\bullet}^1.$$
In this case, we denote
\[\Vert\mathscr{U}_R\Vert_{\gamma, \bullet}=\inf\{M'>0~ \colon
\mathscr{U}_{ R} \text{ is } (M',\gamma)\text{-H\"older continuous with respect to the  $\bullet$-topology}\}.\]
\end{enumerate}
\end{definition}

\begin{lemma}\label{lem:product_superpotential} Let $2\leq s\leq k$ be an integer.
{Let $\bullet$ be either $*$ or $\bigstar$.}
Let  $Q\in \mathscr{D}^0_{k-s+1}(X)$ whose super-potential satisfies $\|\mathscr{U}_Q\|_{\gamma,\bullet}\leq1$. 
There exists $A_{1}>0$,
independent of $Q$ and $\gamma$,
such that
$\|\mathscr{U}_{\omega\wedge Q}\|_{\gamma,\bullet}\leq A_1^{\gamma}$.
\end{lemma}
\begin{proof}
Take $S\in \mathscr{D}_{s-1}^0(X)$ with $\|S\|_\bullet\le1$.
For $\bullet=*$, we have
$\|S\wedge\omega\|_*\leq \|S\|_*$.
For $\bullet=\bigstar$, consider an arbitrary decomposition
$S=S_1-S_2$ into positive currents.
By definition, {we have}
\begin{align*}
\|S\wedge\omega\|^\bigstar
&\leq \sup_{\Phi\in\mathscr{H}_{k-s}}
\big(\|S_1\wedge\omega\wedge\Phi\|
+\|S_2\wedge\omega\wedge\Phi\|\big)\\
&=\sup_{\Phi\in\mathscr{H}_{k-s}}
\|\Phi\wedge\omega\|
\left(
\left\|S_1\wedge
\frac{\omega\wedge\Phi}{\|\Phi\wedge\omega\|}\right\|
+
\left\|S_2\wedge
\frac{\omega\wedge\Phi}{\|\Phi\wedge\omega\|}\right\|
\right)\\
&\leq \sup_{\Phi'\in\mathscr{H}_{k-s+1}}
\big(\|S_1\wedge\Phi'\|+\|S_2\wedge\Phi'\|\big),
\end{align*}
where we set
{$\Phi':=
\omega\wedge\Phi / \|\Phi\wedge\omega\|$}
and use the fact that
$\|\Phi\wedge\omega\|\leq 1$
for every $\Phi\in\mathscr{H}_{k-s}$.
Taking the infimum over all such decompositions
$S=S_1-S_2$, we obtain
$\|S\wedge\omega\|^\bigstar\leq\|S\|^\bigstar$.
Thus, in both cases,
$\|S\wedge\omega\|_\bullet\leq\|S\|_\bullet$.
Since $\{Q\}=0$, we have
$\mathscr{U}_{\omega\wedge Q}(S)
=
\mathscr{U}_Q(\omega\wedge S)$.
Moreover,
\[
\|\omega\wedge S\|_{\mathcal C^{-2}}
=
\sup_{\|\Phi\|_{\mathcal C^2}\le1}
|\langle { S},\omega\wedge\Phi\rangle|
\le
A_{1}\,\|S\|_{\mathcal C^{-2}},
\]
where $A_{1}$ depends only on $X$.
Hence, we deduce from the assumption on $\mathscr U_Q$ that
$|\mathscr{U}_{\omega\wedge Q}(S)|
\le
 A^\gamma_{1}
\|S\|_{\mathcal C^{-2}}^\gamma$.
This completes the proof.
\end{proof}
\begin{lemma}\label{lem:bigstar-bounded-functional}
Let $F:\mathscr D_q^0(X)\to\mathbb R$ be linear
and continuous with respect to the $\bigstar$-topology.
Then
\[
M_F:=\sup_{\substack{R\in\mathscr D_q^0(X)\\
\|R\|^{\bigstar}\leq1}}|F(R)|<\infty.
\]
In particular, every $\bigstar$-continuous
super-potential is bounded on the
$\bigstar$-unit ball
 of $\mathscr D_q^0 (X)$.
\end{lemma}
\begin{proof}
Let $\Omega$ be a $(k-q,k-q)$-test form with $\|\Omega\|_{\infty}\le 1$.  
There exists a
constant $A > 0$ independent of $\Omega$ so that
$|\Omega|\leq A \|\Omega\|_\infty \omega^{k-q}$. For any decomposition
$R=R^+-R^-$ into positive currents, positivity gives
\begin{align*}
|\langle R,\Omega\rangle|
&\leq |\langle R^+,\Omega\rangle|
     +|\langle R^-,\Omega\rangle|\leq A\,\|\Omega\|_{\infty}
\big(\|R^+\wedge\omega^{k-q}\|
+\|R^-\wedge\omega^{k-q}\|\big)\\
&\leq A\,\|\Omega\|_{\infty}
\sup_{\Phi\in\mathscr H_{k-q}}
\big(\|R^+\wedge\Phi\|+\|R^-\wedge\Phi\|\big).
\end{align*}
Taking the infimum over all such decompositions,
we obtain
$|\langle R,\Omega\rangle|
\leq A\|\Omega\|_{\infty}
\|R\|^{\bigstar}$.
Suppose, by contradiction, that $M_F=\infty$.
Then there exist $R_j\in\mathscr D_q^0(X)$
with $\|R_j\|^{\bigstar}\leq1$ such that
$a_j:=|F(R_j)|\to\infty$.
Set $V_j:=R_j/a_j$. Then
$\|V_j\|^{\bigstar}\leq a_j^{-1}\to0$.
By the above estimate,
$\langle V_j,\Omega\rangle\to0$
for every smooth test form $\Omega$.
Hence $V_j\to0$ weakly.
Since their $\bigstar$-norms are uniformly bounded,
$V_j\to0$ in the $\bigstar$-topology.
By continuity, $F(V_j)\to F(0)=0$,
contradicting $|F(V_j)|=1$. Hence, $M(F)<\infty$. The second assertion follows immediately by taking
$F=\mathscr U_S$, where $\mathscr U_S$ is a
$\bigstar$-continuous super-potential.  
\end{proof}

We conclude this subsection with a simple cohomological decomposition that will be used
in 
Sections \ref{ss:new-speed}--\ref{subsection:4.3} { and Appendix \ref{appendix:Regularization}}.
\begin{lemma}
\label{lem:decom_S=alpha+S'}
For every $S \in \mathscr D_q$
there exists a smooth closed form $\widetilde S \in \mathscr D_q$, called the
\emph{standard form associated with $S$}, such that:
\begin{enumerate}
\item $\{S\}=\{\widetilde S\}$ and $\|\widetilde S\|_* \lesssim \|S\|_*$;
\item 
setting
$S^0:=S-\widetilde S$, then $S^0 \in \mathscr D_q^0$ and $\|S^0\|_* \lesssim \|S\|_*$;
\item setting
$M_{\widetilde S}:=\sup_{\|\Omega\|^\bigstar \le 1} |\mathscr U_{\widetilde S}(\Omega)|$,
we have $M_{\widetilde S} \lesssim \|S\|_*$.
\end{enumerate}
All the
implicit constants are independent of $S$.
\end{lemma}

\begin{proof}
Recall that, by the notation fixed above, $\alpha_{{k-q}}=(\alpha_{ k-q,1},\dots,\alpha_{ k-q,h(k-q)})$
is a family of smooth closed real $( k-q,k-q)$-forms such that
$\{\alpha_{ k-q,1}\},\dots,\{\alpha_{ k-q,h(k-q)}\}$ form a basis of $H^{ k-q,k-q}(X,\mathbb R)$, and
$\check\alpha_{ k-q}=(\check\alpha_{ k-q,1},\dots,\check\alpha_{ k-q,h(k-q)})$
is the dual basis in $H^{q,q}(X,\mathbb R)$ with respect to Poincar\'e duality.
Write
$\{S\}=\sum_{i=1}^{h( k-q)} s_i \{{\check{\alpha}_{k-q,i}}\}$,
where
$s_i=\langle S,{ \alpha_{k-q,i}}\rangle$.
Since the forms ${\alpha_{k-q,i}}$ are fixed and smooth, we have
$|s_i| \lesssim \|S\|_*$
for every 
$1\le i\le h( k-q)$,
where the implicit constants are
independent of $S$ and can be taken independent of $i$.

Define
$\widetilde S:=\sum_{i=1}^{h( k-q)} s_i\, \check\alpha_{k-q,i}$.
Then $\widetilde S$ is a smooth closed $(q,q)$-form and, by construction,
we have $\{\widetilde S\}=\{S\}$. Moreover,
 we have
$\|\widetilde S\|_*
\le \sum_{i=1}^{h( k-q)} |s_i|\,\|{\check\alpha_{k-q,i}}\|_*
\lesssim \|S\|_*$,
which proves (1).

Set $S^0:=S-\widetilde S$ as in the statement. Since $\{S^0\}=
\{S\}-\{\widetilde S\}=0$, we have $S^0\in \mathscr D_q^0$, and by the triangle
inequality, we have
$\|S^0\|_* \le \|S\|_* + \|\widetilde S\|_* \lesssim \|S\|_*$,
where the implicit constant is independent of $S$.
This proves (2).

Finally, since each $\check\alpha_{k-q,i}$ is smooth, its super-potential $\mathscr U_{\check\alpha_{k-q,i}}$ is continuous
with respect to the $\bigstar$-topology.  Hence, 
by Lemma~\ref{lem:bigstar-bounded-functional},
$\mathscr U_{\check\alpha_{k-q,i}}$ is bounded on the $\bigstar$-unit ball
$\{\Omega\in \mathscr D_{k-q+1}^0 : \|\Omega\|^\bigstar\le 1\}$. Set
\[
m_i:={\sup_{\|\Omega\|^\bigstar\le 1}} |\mathscr U_{\check\alpha_{k-q,i}}(\Omega)| < +\infty.
\]
By linearity of super-potentials, we obtain
\[
M_{\widetilde S}
= {\sup_{\|\Omega\|^\bigstar\le 1}}
\left|\sum_{i=1}^{h( k-q)} s_i\,\mathscr U_{{\check\alpha_{k-q,i}}}(\Omega)\right|
\le \sum_{i=1}^{h( k-q)} |s_i|\,m_i
\lesssim \|S\|_*.
\]
This proves (3) and completes the proof.
\end{proof}

\subsection{Skoda-type estimate}
 We will need the following version 
of the ``Skoda-type'' estimate
given in \cite[Proposition 2.1]{DS10} for the $*$-topology.
Let $1\leq s\leq k$ be an integer.
Whenever $Q$ 
has a continuous super-potential, we interpret 
$\mathscr U_S(Q)$,
for $S\in\mathscr D_s^0(X)$, 
as $\mathscr U_Q(S)$,
extending the usual pairing by symmetry.

\begin{proposition} \label{prop:U_S-estimates} 
Take ${ Q}\in \mathscr{D}_{k-{s}+1}^{0}$ with $\|{ Q}\|^{\bigstar} \leq 1$
and $\|\mathscr U_{ Q} \|_{\gamma, *}\leq M$.
 There exists
 a constant $A>0$ independent of $Q$, $\gamma$, and $M$ such that the super-potential $\mathscr{U}_{S}$ of $S$ satisfies
$$
\left|\mathscr{U}_{S}({Q})\right| \leq A\left(1+\gamma^{-1} \log ^{+} M\right)
$$ 
for any current $S\in\mathscr{D}_{{ s}}^{0}$ with $\|S\|_{*} \leq 1$, where $\log ^{+} (\cdot):=\max (0, \log(\cdot) )$.
\end{proposition}

\begin{proof} 
We adapt the proof of \cite[Proposition 2.1]{DS10} to the use of the $\bigstar$-norm.
Multiplying $S$ by a constant  allows us to assume $\|S\|_*\leq c^{-k-{4}}$, where $c$ is
the maximum of the constants given
by Proposition \ref{prop:property-transform}, 
Lemma \ref{lem:L-infty}, Proposition \ref{prop:star-bigstar-ddc}, Proposition \ref{prop:smooth-bigstar}, and Proposition \ref{prop:L_1-potential}.
We define  inductively
\begin{align*}
    S_0:=S \quad \text{and} \quad S_{i+1}:=\mathcal{L}_{\infty}(S_i)
\quad \text{for}\quad 0\leq i\leq k+1.
\end{align*}
 For $0\le i\le k+2$, define also
\[
u_i(\theta):=\mathscr U_{\mathcal L_\theta(S_i)}(Q)
\quad \text{and} \quad m_i:=u_i(0).
\]
These expressions are well-defined since $Q$ has a
continuous super-potential. For smooth $Q$, 
they agree
with $\mathscr U_{\mathcal L_\theta(S_i)}(Q)$.
Moreover, we have
 $m_i=u_{i-1}(\infty)$ for $1\le i\le k+2$.
Proposition \ref{prop:property-transform} (5)
implies that $\|S_i\|_*\leq 1/c^{2}$. 

When ${Q}$ is smooth,
we have $\mathscr{U}_{\mathcal{L}_{\theta}(S)}({ Q})=\mathscr{U}_S(\mathcal{L}_{\theta}({Q}))$,
see \cite[Lemma 3.2.5]{DS10JAG}. For such ${Q}$,  applying  Proposition \ref{prop:star-bigstar-ddc}  we obtain $$\|dd^cu_i\|_*=\Vert dd^c_{\theta}\mathscr{U}_{\mathcal{L}_{\theta}(S_{i})}({ Q})\Vert_*\leq c\Vert S_{i}\Vert_*\Vert { Q}\Vert^{\bigstar}\leq 1. $$ 
For a
general ${Q}$, we use an approximation.
Indeed, by 
Proposition \ref{prop:smooth-bigstar} smooth currents are dense in $\mathscr{D}_{k-{s}+1}^0$  with respect to the  $\bigstar$-topology. 
We briefly explain how to pass to the limit.  Since \(\mathscr U_Q\) is Hölder continuous, it is also continuous with respect to the \(*\)-topology.
By Proposition \ref{prop:smooth-bigstar},  we may choose smooth $Q_n\in \mathcal{D}_{k-s+1}^0$ such that $\|Q_n\|^\bigstar\leq c\|Q\|^\bigstar$ where \(c>0\) is a constant independent of \(Q\) and \(n\). By the second assertion of Proposition \ref{prop:smooth-bigstar}, \(Q_n\) converges \(SP^*\)-uniformly to \(Q\). 
Using this and  Proposition \ref{prop:property-transform}(3)-(5), we obtain
\begin{align*}
|u_{i,n}(\theta)-u_i(\theta)|&=|\mathscr{U}_{\mathcal{L}_\theta(S_i)}(Q_n)-\mathscr{U}_{\mathcal{L}_\theta(S_i)}(Q)|=
|\mathscr{U}_{Q_n}(\mathcal{L}_\theta(S_i))-\mathscr{U}_{Q}(\mathcal{L}_\theta(S_i))|\\&
\leq \|\mathcal{L}_\theta(S_i)\|_*\sup_{\|T\|_*\leq 1}|\mathscr{U}_{Q_n}(T)-\mathscr{U}_{Q}(T)|
\\&\leq c\|S_i\|_*\sup_{\|T\|_*\leq 1}|\mathscr{U}_{Q_n}(T)-\mathscr{U}_{Q}(T)|.
\end{align*}
 It follows that  $\sup_{\theta\in\mathbb{P}^1}|u_{i,n}(\theta)-u_i(\theta)|\to 0$. 
Consequently, $dd^c_\theta u_{i,n}(\theta)\to dd^c_\theta u_i(\theta)$ in the sense of currents on $\mathbb{P}^1$. Since the \(*\)-norm of signed measures is lower semicontinuous under weak convergence, we obtain
\begin{align*}
\|dd^c_\theta u_i\|_*
\leq\liminf_{n\to\infty}
\|dd^c_\theta u_{i,n}\|_*=\liminf_{n\to\infty}\Vert dd^c_{\theta}\mathscr{U}_{\mathcal{L}_{\theta}(S_{i})}({Q_n})\Vert_*\leq c^2\|S_i\|_*\|Q\|^\bigstar\leq 1.
\end{align*} 

\medskip

Thus, the above 
extends by continuity to any
${Q}$ with a continuous super-potential, as in the statement.

It follows from Lemma
\ref{lem:L-infty} that
$
\|S_{k+2}\|_{L^{\infty}}\leq 1$. More precisely, we have
\begin{align*}
\|S_{k+2}\|_{L^{\infty}}&=\|\mathcal{L}_{\infty}(S_{k+1})\|_{L^{\infty}} & \\
& \leq c\|S_{k+1}\|_{L^{k+1}} & \text{by Lemma \ref{lem:L-infty} (3)}\\
& =c\|\mathcal{L}_{\infty}(S_{k})\|_{L^{k+1}}=c\|\mathcal{L}_{\infty}(S_{k})\|_{L^{q_{k}}}\\
&
\leq c^2\|S_{k}\|_{L^{q_{k-1}}}
& {\text{by Lemma \ref{lem:L-infty} (2)}}\\
& =c^2\|\mathcal{L}_{\infty}(S_{k-1})\|_{L^{q_{k-1}}}\\
& \leq c^3\|S_{k-1}\|_{L^{q_{k-2}}}
 \leq \dots \leq c^k\|\mathcal{L}_{\infty}(S_{ 1})\|_{L^{q_1}}
 & {\text{by Lemma \ref{lem:L-infty} (2)}}\\
 & \leq c^{k+1}\|S_{1}\|_*=c^{k+1}\|\mathcal{L}_{\infty}(S)\|_*  & {\text{by Proposition \ref{prop:property-transform} (5)}}\\
 & \leq c^{k+2}\|S\|_*\leq 1/c^2<1,
\end{align*}
where in the last step we applied the assumption
$\|S\|_*\leq c^{-k-4}$. 
By Proposition \ref{prop:L_1-potential}, ${Q}$ admits a Green potential $U_{Q}$ with
$\Vert U_{Q}\Vert_{L^1}\leq c$.
By definition, we have $m_{k+2}=\langle S_{k+2},U_{Q}\rangle$.  Thus, we have $|m_{k+2}|\leq c$.

We need to show that $\left|m_{0}\right| \leq A\left(1+\gamma^{-1} \log ^{+} M\right)$ for some constant $A>0$. 
 Since replacing $M$ by a larger number preserves the assumptions, we may assume that $M>1$.
It is enough to check that $|m_{i}- m_{i+1}| \leq A\left(1+\gamma^{-1} \log M\right)$ for some constant $A>0$. We have $m_{i}-m_{i+1}=v_{i}(0)$ where $v_{i}:=u_{i}-m_{i+1}$. The above properties of $u_{i}$ imply that 
the $v_{i}$'s
are continuous, vanish outside the unit disc, and satisfy $\left\|dd^c v_{i}\right\|_* \leq 1$. By
classical Skoda-type estimates (see, for instance, \cite{Sko72, DS10JAG}),
we have $\left\|e^{\lambda\left|v_{i}\right|}\right\|_{L^{1}\left(\mathbb{P}^{1}\right)} \leq \beta$ for some universal constants $\beta>0$ and $\lambda>0$. We then deduce that there exists 
$\theta$ satisfying $|\theta| \leq M^{-1 / \gamma}$ and $\left|v_{i}(\theta)\right| \leq(A-1)+A \gamma^{-1} \log M$ for a fixed constant $A$ large enough.
 By Proposition \ref{prop:property-transform} (3) and (5), and our initial
normalization, the current
$S_i-\mathcal L_\theta(S_i)$ is cohomologous to zero
and satisfies
\[
\|S_i-\mathcal L_\theta(S_i)\|_*
\le (1+c)\|S_i\|_*
\le (1+c)c^{-2}\le 1,
\]
where we assume $c\ge 2$.
Therefore, by the Hölder continuity of $\mathcal U_Q$
and Proposition \ref{prop:property-transform} (6), we obtain
$$
\begin{aligned}
\left|v_{i}(0)-v_{i}(\theta)\right| & =\left|\mathscr{U}_{S_{i}}({Q})-\mathscr{U}_{\mathcal{L}_{\theta}\left(S_{i}\right)}({ Q})\right|=\left|\mathscr{U}_{{Q}}\left(S_{i}\right)-\mathscr{U}_{{Q}}\left(\mathcal{L}_{\theta}\left(S_{i}\right)\right)\right| \\
& \leq M \operatorname{dist}_{2}\left(S_{i}, \mathcal{L}_{\theta}\left(S_{i}\right)\right)^{\gamma} \leq M|\theta|^{\gamma} \leq 1.
\end{aligned}
$$
Therefore, $\left|v_{i}(0)\right| \leq A\left(1+\gamma^{-1} \log M\right)$. This completes the proof of Proposition \ref{prop:U_S-estimates}.
\end{proof}

\section{Action of automorphisms and convergence towards Green currents}
\label{s:action}

Throughout this section, $(X,\omega)$ denotes a compact K\"ahler manifold of dimension $k$, and $f\colon X\to X$ is a holomorphic automorphism.

\subsection{Action of automorphisms and dynamical degrees}\label{subsection:Action}
For every integer $ 0\leq q\leq k$ the \textit{$q$-th dynamical degree} of $f$ is defined  as 
$$d_q(f):=\lim_{n\to\infty}\left(\int_{X}(f^n)^*(\omega^q)\wedge\omega^{k-q}\right)^{1/n}.$$ 
Results
 by Khovanskii \cite{Khovanskii}, Teissier \cite{Teissier}, and Gromov \cite{Gromov90} imply
  that the sequence ${q\mapsto\log{d_q}}$ is concave. So, there are integers $0\le p\le p'\le k$ such that
$1=d_0<\dots<d_p=\dots=d_{p'}>\dots>d_k=1$.
 Throughout this paper
 we shall assume that $f$ has a \emph{simple action on cohomology}. 
 This means that  $p=p'$ and
 that the action of $f^*$ on $H^{p,p}(X,\mathbb{R})$ 
 has a unique eigenvalue of maximal modulus, denoted by $d_p$, which we call the \emph{main dynamical degree} of $f$.
We
define $\delta_f$ to be the maximum between 
$\max\{d_{p-1},d_{p+1}\}$
 and the absolute values of
  all the eigenvalues different from $d_p$
  for the action of $f^*$
 on $H^{p,p}(X,\mathbb{R})$. We have $\delta_f<d_p$ by definition.
Throughout the text, we shall refer to any $\delta$
with $\delta_f < \delta < d_p$  
as an \emph{auxiliary dynamical degree} for $f$. 

\medskip

Fix $\varepsilon>0$.
For every $0\leq q \leq k$
 define $\delta_q := d_q+\varepsilon$ if
 $q\neq p$, and  $\delta_p := d_p$. 
When given
 an auxiliary
 degree $\delta$,
 we will always assume that 
  $\varepsilon$ is small enough so that $d_q+\varepsilon<\delta$
for every $q\neq p$.
 Since the mass of a positive closed current can be computed cohomologically, for every $R\in\mathscr{D}_q(X)$ we have 
\begin{equation}\label{eq:growth-*norm}
   \Vert(f^n)^*(R)\Vert_*\leq C\delta_q^n\Vert R\Vert_*
\end{equation} for some constant $C>1$
which is independent of $R$ and $n$
(see \cite[Lemma 4.1.2]{DS10JAG}). 
The following lemma shows that the dynamical
degrees control the growth
of the pull-backs of currents 
also with respect
to the $\bigstar$-norm.
\begin{lemma}\label{lem:growth}
   For every $R\in \mathscr{D}_q(X)$ and every integer $n\in\mathbb{N}$ we have
   $$\Vert(f^n)^*(R)\Vert^{\bigstar}\leq C\delta_q^n\Vert R\Vert^{\bigstar}$$
for some constant $C>1$ which is
 independent of $R$ and $n$.
\end{lemma}
\begin{proof} 
Recall that for any  $\Phi\in \mathscr{H}_{k-q}$ we have $\|\Phi\|\leq 1$.
    By definition, for every $\kappa>0$ there exist  positive currents $R_1$ and $R_2$ such that $R=R_1-R_2$ and $$\sup_{\Phi\in \mathscr{H}_{k-q}}(\Vert R_1\wedge \Phi\Vert+\Vert R_2\wedge \Phi\Vert)<\Vert R\Vert^{\bigstar}+\kappa.$$ 
    Since 
    every
    $\Phi\in\mathscr{H}_{k-q}$ is a positive closed $(k-q,k-q)$-current,
    by (\ref{eq:growth-*norm})  (applied to $f^{-1}$)
    there exists a constant $C>1$,
     independent of $\Phi$ and $n$, such that
$ \|(f^n)_*\Phi\|\leq C\,\delta_q^n\,\|\Phi\| $.
    For every integer $n\geq 1$ we   decompose $(f^n)^*(R)=(f^n)^*(R_1)-(f^n)^*(R_2)$. Denoting
    $\Phi':=(f^n)_*(\Phi)/ \|(f^n)_*(\Phi)\|$  we have
    \begin{align*}
       \Vert(f^n)^*(R)\Vert^{\bigstar}
       &\leq\sup_{\Phi\in\mathscr{H}_{k-q}}(\Vert(f^n)^*(R_1) \wedge \Phi\Vert+\Vert(f^n)^*(R_2)\wedge \Phi\Vert)\\
       &=\sup_{\Phi\in\mathscr{H}_{k-q}} \Vert (R_1+R_2)\wedge (f^n)_*(\Phi)\Vert
       =
    \sup_{\Phi\in\mathscr{H}_{k-q}}( \Vert(f^n)_*(\Phi)\Vert\cdot \Vert(R_1+R_2)\wedge\Phi'\Vert)\\
    &\leq
    \sup_{\Phi\in\mathscr{H}_{k-q}}( C\delta_q^n\Vert\Phi\Vert\cdot \Vert(R_1+R_2)\wedge\Phi'\Vert)\\
    &
    \leq C\delta_q^n\sup_{ { \Phi''}\in\mathscr{H}_{k-q}}(\Vert R_1\wedge { \Phi''}\Vert+\Vert R_2\wedge { \Phi''}\Vert)
    < C\delta_q^n(\Vert R\Vert^{\bigstar}+ \kappa).
    \end{align*}
As
$\kappa$ is arbitrary, the assertion follows.
\end{proof}
The following lemma provides a uniform growth estimate for wedge products of pull-backs of $\omega$.

\begin{lemma}\label{lem:pull-back-mass}
Let $\delta$ be an auxiliary
degree for $f$.
There exists a constant $A>0$ such that for every $1\le q<p$ and all integers
$n_1\ge n_2\ge \dots \ge n_q\ge 0$
we have
\[
\left\|
\bigwedge_{i=1}^{q}(f^{n_i})^*(\omega)
\right\|
\le A\,\delta^{n_1}.
\]
\end{lemma}

\begin{proof}
For $q=1$, the statement follows directly from the definition of $\delta$ and the fact that $q<p$.
Assume the result holds for $q-1\ge1$. Thus there exists $A_{q-1}>0$, independent of the $n_j$'s, 
such that
\[
\|\Omega\|\le A_{q-1}\,\delta^{\,n_1-n_q},
\qquad
\text{where}
\qquad 
\Omega :=
\bigwedge_{i=1}^{q}(f^{\,n_i-n_q})^{*}(\omega).
\]
We then compute
\begin{align*}
\left\|
\bigwedge_{i=1}^{q}(f^{n_i})^*(\omega)
\right\|
&=
\|(f^{n_q})^*\Omega\|
=
\|\Omega\|\,
\left\|
(f^{n_q})^*\!\left(\frac{\Omega}{\|\Omega\|}\right)
\right\| \\
&\le
c\,\|\Omega\|\,\delta^{n_q}
\le
cA_{q-1}\,\delta^{n_1-n_q}\delta^{n_q}
=
cA_{q-1}\,\delta^{n_1},
\end{align*}
where $c>0$ is independent of $\Omega$ and $n_q$.
This completes the induction.
\end{proof}

Finally, 
we will need 
the following
explicit dependence on $\gamma$ 
of the estimate 
in \cite[Lemma 3.2]{DS10}.
 A similar statement holds with $\|\cdot\|_{\gamma,\bigstar}$ instead of $\|\cdot\|_{\gamma, *}$.

\begin{lemma}\label{lem:Holder_constant}
Fix $1\leq s\leq k$
and $Q\in \mathscr{D}^0_{k-s+1}$. Assume 
 $\|\mathscr U_Q\|_{\gamma,*}\leq 1$ for some
$0<\gamma\le1$. There exist constants $A>1$ and $c>1$, independent of $Q$
 and $\gamma$,
such that for every $n\in\mathbb{N}$ we have
\[
\|\mathscr U_{(f^n)_* Q}
\|_{\gamma,*}
\leq A^{n\gamma}(c\,\delta_s^n)^{1-\gamma}.
\]
\end{lemma}

\begin{proof}
Take $S\in \mathscr{D}_s^0(X)$ with $\|S\|_*\le1$. By the H\"older continuity
of $\mathscr U_Q$, we have
$|\mathscr{U}_Q(S)|\le \|S\|_{\mathcal C^{-2}}^\gamma$.
Since $\{Q\}=0$,
 we also have
\[
\mathscr{U}_{(f^n)_*Q}(S)
=
\mathscr{U}_Q((f^n)^*S).
\]
The above identity is clear when $S$ is smooth. 
 Since $\mathscr{U}_Q$  is continuous and smooth forms are dense in $\mathscr{D}_s^0(X)$, we deduce
 that $\mathscr{U}_{(f^n)_*Q}$ is continuous, hence 
  the above identity holds
for every $S\in\mathscr{D}_s^0(X)$.

By definition of dynamical degree there exists $c>1$ such that
$\|(f^n)^*S\|_*\le c\,\delta_s^n\|S\|_*$ for all $n$.
Hence $\|(f^n)^*S/(c\,\delta_s^n)\|_*\le1$ and
we obtain
\[
{ \big|}\mathscr{U}_{(f^n)_*Q}(S){\big|}
=
{\big|}\mathscr{U}_Q((f^n)^*S){\big|}
=
c\,\delta_s^n\,
{\big|}\mathscr{U}_Q\!\left(\frac{(f^n)^*S}{c\,\delta_s^n}\right){\big|}
\le
(c\,\delta_s^n)^{1-\gamma}
\|(f^n)^*S\|_{\mathcal C^{-2}}^\gamma .
\]
It remains to control the $\mathcal C^{-2}$-norm in the last term.
Since $f^{-1}$ is smooth, there exists $A>1$ such that
$\|f_*\Phi\|_{\mathcal C^2}\le A\|\Phi\|_{\mathcal C^2}$ for all test forms $\Phi$.
Therefore
\[
\|f^*S\|_{\mathcal C^{-2}}
=
\sup_{\|\Phi\|_{\mathcal C^2}\le1}
|\langle S,f_*\Phi\rangle|
\le
A\,\|S\|_{\mathcal C^{-2}} .
\]
Iterating gives $\|(f^n)^*S\|_{\mathcal C^{-2}}\le A^n\|S\|_{\mathcal C^{-2}}$.
Substituting into the previous estimate yields
\[
|\mathscr U_{(f^n)_*Q}(S)|
\le
A^{n\gamma}(c\,\delta_s^n)^{1-\gamma}
\|S\|_{\mathcal C^{-2}}^\gamma,
\]
which proves the assertion.
\end{proof}

\subsection{Exponential convergence of super-potentials}
 \label{ss:new-speed}
The goal of this section is to 
prove
Proposition \ref{prop: Sp-convergence}.
It
is a version of \cite[Proposition 4.2.2]{DS10JAG} and
establishes
exponential convergence of the super-potentials of the normalized pull-backs of the standard 
form
associated with $S$ in Lemma~\ref{lem:decom_S=alpha+S'}, uniformly on $\bigstar$-bounded sets. 
It will be a key ingredient in the proof of Theorem \ref{thm:new-speed}.

Let $\T^+$ and $\T^-$ be the Green currents of $f$ and $f^{-1}$,
of bidegrees $(p,p)$ and $(k-p,k-p)$, respectively.
They satisfy $f^*\T^+=d_p\T^+$ and $f_*\T^-=d_p\T^-$
and are unique in their respective cohomology classes
\cite[Section~4.3]{DS10JAG}.
Moreover, for any current 
$S\in\mathscr{D}_p(X)$,
 there exists a constant $c=c_S$
  such that $d_p^{-n}(f^n)^*(S)$
converges to $c\T^+$.

We say that 
a convergence
   $S_n\to S$ of currents in $\mathscr{D}_p$
  is $SP^{\bigstar}$-\emph{uniform}
    if
 $\mathscr{U}_{{S}_{n}}$ converges
 to $\mathscr{U}_{{S}}$ uniformly on any $\bigstar$-bounded set of smooth forms in $\mathscr{D}_{k-{p}+1}^{0}$. 
Recall that 
we always assume that super-potentials are normalized as in Section \ref{subsection:superpotentials}.
Observe also
that, by linearity, it suffices to verify the $SP^{\bigstar}$-uniform convergence on the unit ball of $\mathscr{D}^0_{k-p+1}(X)$.

 \begin{proposition}
 \label{prop: Sp-convergence} 
Let $f:X\to X$ be a holomorphic automorphism of $(X,\omega)$ with simple action on cohomology.
Take $S\in \mathscr D_p\cap \mathscr D^1_{p,*}$, and let $\widetilde S\in \mathscr D_p$ be the standard form associated with $S$ as in Lemma 
\ref{lem:decom_S=alpha+S'}.
Assume that $d_p^{-n}(f^n)^*\{\widetilde S\}\to \underline{c}\in H^{p,p}(X,\mathbb R)$.
Then
$\widetilde S_n:=d_p^{-n}(f^n)^*(\widetilde S)$
converges
$SP^{\bigstar}$-uniformly
to a current $\T_{\underline c}\in \mathscr D_p$
which depends only on $\underline c$.
Moreover, for every $\delta_f<\delta<d_p$, we have
\[
|\mathscr U_{\widetilde S_n}(Q)-\mathscr U_{\T_{\underline {c}}}(Q)|
\lesssim
\left(\frac{\delta}{d_p}\right)^n \|Q\|^\bigstar
\qquad \text{for every } Q\in \mathscr D_{k-p+1}^0,
\]
where the implicit constant is independent of $S$, $Q$ 
and $n$.

\end{proposition}
We will deduce 
Proposition~\ref{prop: Sp-convergence} from the 
following version of
\cite[Proposition 4.2.2]{DS10JAG}
for the $\|\cdot\|^\bigstar$ norm.

\begin{proposition}\label{prop: Sp-convergence-1} Let $f:X\to X$ be a holomorphic automorphism of
 $(X,\omega)$
 with simple action on cohomology. 
Take $S\in \mathscr{D}_{p}$ with a continuous super-potential with respect to the $\bigstar$-topology.  Assume that $d_{p}^{-n}\left(f^{n}\right)^{*}\{S\}\to \underline{c}\in H^{p, p}(X, \mathbb{R})$. Then $S_n:=d_{p}^{-n}\left(f^{n}\right)^{*}(S)$ converges $SP^{\bigstar}$-uniformly to a current $\T_{\underline{c}}\in\mathscr{D}_{p}$ which depends only on $\underline{c}$.
Moreover, 
for every $\delta_f < \delta<d_p$
we have
\[|\mathscr{U}_{S_n}({Q})-\mathscr{U}_{\T_{\underline{c}}}({Q})|
\lesssim \left(\frac{\delta}{d_p}\right)^n{(\|S\|_*+}M_S )\,\|{ Q}\|^{\bigstar}
\qquad
\text{for every }  
{Q} \in  \mathscr{D}_{k-p+1}^0,
\]
where $M_S = {\sup_{\|\Omega\|^{\bigstar} \leq 1} }|\mathscr{U}_S(\Omega)|$ and the implicit constant is independent of $S, \underline{c},$ and $Q$.

 \end{proposition} 
 \begin{proof}
      As in Section \ref{subsection:superpotentials}, we
 let $\alpha_{p}=\left(\alpha_{p,1}, \ldots, \alpha_{p,h(p)}\right)$ be a family of smooth closed real ($p, p$)-forms such that $\{\alpha\}=\left\{\{\alpha_{p,1}\}, \ldots,\{\alpha_{p,h(p)}\right\}\}$ is a basis of $H^{p, p}(X, \mathbb{R})$,
  where $h(p)$ is the dimension of $H^{p, p}(X, \mathbb{R})$. We always consider
  super-potentials normalized by $\alpha_p$.
For simplicity, as $p$ is fixed, we will denote
$h=h(p)$ and
$\alpha = \alpha_p 
=(\alpha_1, \dots, \alpha_h)$
in what follows.
  Let $M$ denote the $h \times h$ matrix whose column of index $j$ is given by the coordinates of $f^{*}\{\alpha_{j}\}$ with respect to the basis $\{\alpha\}$. Let $\mathscr{U}_{j}$ denote the super-potential of $f^{*}\left(\alpha_{j}\right)$ and define $\mathscr{U}:=\left(\mathscr{U}_{1}, \ldots, \mathscr{U}_{h}\right)$. Let $A={ }^{t}\left(a_{1}, \ldots, a_{h}\right)$ denote the coordinates of $\{S\}$ in the basis $\{\alpha\}$ and $\mathscr{U}_{S}, \mathscr{U}_{S_{n}}$ the super-potentials of $S$ and of $S_{n}:=d_p^{-n}\left(f^{n}\right)^{*}S$, respectively.

We mainly
follow the arguments of \cite[Proposition~4.2.2]{DS10JAG}, 
adapting them to the use of the $\bigstar$-topology.
 By \cite[Proposition 4.1.1]{DS10JAG} the limit $\underline{c}$ of $ d_{p}^{-n}\left(f^{n}\right)^*\{S\}$ is a class in $\{\T^+\}$.
 Write $\underline{c}={ }^{t}\left(c_{1}, \ldots, c_{h}\right)$ with respect to the basis $\{\alpha\}$. Then, $ d_{p}^{-n} M^{n} A$ converges to $\underline{c}$  , and therefore $M\underline{c}=d_p\underline{c}$.
Following \cite[Lemma 4.2.3]{DS10JAG}, the super-potential $\mathscr{U}_{{S_n}}$ of $S_n$ is equal to
\begin{align}
\mathscr{U}_{{ S_n}} 
\label{eq:summ_superpotential}
& =\sum_{l=0}^{n-1}\left(\mathscr{U} \circ { \left(f^{l}\right)_{*}}\right) \frac{M^{n-l-1} A}{ d_{p}^{n}}+ d_{p}^{-n} \mathscr{U}_{S} \circ { \left(f^{n}\right)_{*}}. 
\end{align}

Since $\mathscr U_S$ is continuous on $\bigstar$-bounded sets, 
the constant 
 $\displaystyle M_{S}=\sup_{\|\Omega\|^{\bigstar}\leq 1}|\mathscr{U}_{S}(\Omega)|$
is finite and depends only
on $S$ (see Lemma~\ref{lem:bigstar-bounded-functional}).
 Using this
 and applying Lemma \ref{lem:growth} with $q=k-p+1{,R=Q}$,
and $f^{-1}$ instead of $f$ 
 we obtain 
\begin{equation}\label{eq:growth_super-potential}
\left|\mathscr{U}_{S}({\left(f^{n}\right)_{*}}({ Q}))\right| \leq {M_S}\left\|{\left(f^{n}\right)_{*}}({ Q})\right\|^{\bigstar} \lesssim \delta^{n}{ M_S\,}\|{ Q}\|^{\bigstar}, 
\end{equation}
for any $\delta>\delta_f$.
The implicit constant above is independent of $S$, $Q$,
and $n$.

Fix $\delta_0$ such that $\delta_f<\delta_0<\delta$. We use the maximum norm on $\mathbb R^h$ and the
induced operator norm on matrices. Recall that $\left\|M^{n}\right\| \sim d_{p}^{n}$. 
Since each $f^*(\alpha_j)$ is smooth, its super-potential
$\mathscr U_j$ is continuous with respect to the
$\bigstar$-topology. {By Lemma~\ref{lem:bigstar-bounded-functional}, $\mathscr{U}_j$ is bounded on unit exact $\bigstar$-ball}. Denoting  the constant  $\displaystyle K:=\max_{j\in \overline{1,h}}\sup_{\|\Omega\|^\bigstar\leq 1}|\mathscr{U}_{j}(\Omega)|<+\infty$ and analogous estimates as in (\ref{eq:growth_super-potential})
for $\mathscr{U}_{j}$, imply that
\begin{equation}\label{eq:Growth-matrix}
\left\|\frac{M^{n-l-1} A}{ d_{p}^{n}}\right\| \lesssim d_{p}^{-l} { \|A\|}\quad \text { and } \quad\left|\left(
 \mathscr U
\circ {\left(f^{l}\right)_{*}}\right)
{(Q)}\frac{M^{n-l-1} A}{ d_{p}^{n}}\right| \lesssim { K\,} \delta^l_{ 0} d_{p}^{-l}{ \|A\|}
{\cdot \|Q\|^{\bigstar}}.
\end{equation}
 For   $A=^t(a_1,\ldots, a_h)$ we have  
\begin{align}\label{eq:estimate-norm-A}
 \|A\|=\max_{1\leq j\leq h}|a_j|= \max_{1\leq j\leq h} |\langle S,\check{\alpha}_j\rangle|\lesssim\max_{1\leq j\leq h}\|S\|_*\lesssim\|S\|_*
\end{align}
where the implicit constant is independent of $S$.
Since $\|M^n\|\sim d_p^n$, we obtain
\[
\|\underline c\|
=\lim_{n\to\infty}\|d_p^{-n}M^nA\|
\lesssim\|A\|\lesssim\|S\|_*.
\]

Since $\sum_{l \geq 0} \delta^{l} d_{p}^{-l}$ converges, we can apply the
dominated convergence for the series
in (\ref{eq:summ_superpotential}). We obtain the following uniform convergence on $\bigstar$-bounded sets:
$$
\mathscr{U}_{{S_n}}\to
\sum_{l = 0}^\infty\left(\mathscr{U} \circ {\left(f^{l}\right)_{*}}\right) M^{-l-1}\underline{c},
\qquad n\to \infty.
$$
The series above converges absolutely and uniformly on every \(\bigstar\)-bounded set, by (\ref{eq:Growth-matrix}). On the other hand, as recalled at the beginning of this subsection, the currents \(S_n\) converge weakly to a current of the form \(\eta \T^+\), where \(\eta\in\mathbb R\) is determined by the limiting class \(\underline c\). Evaluating the super-potentials on smooth exact test forms, we conclude that the sum of the series is precisely \(\mathcal U_{\eta \T^+}\). Thus \(\T_{\underline c}=\eta \T^+\), and the convergence \(S_n\to \T_{\underline c}\) is $SP^\bigstar$-uniform.
To conclude, we need to show that the convergence above is exponentially fast, as in the
statement. Since $\delta_{0}>\delta_f$, by the
spectral decomposition of \(M\),
and using $M\underline{c}=d_p\underline{c}$ we have
\[
\|d_p^{-n}M^nA-\underline{c}\|{=\|d_p^{-n}M^n(A-\underline{c})  \|}\leq C\left(\frac{\delta_{0}}{d_p}\right)^n{\|A\|}
\]
for some constant $C>0$ independent of $n$  and $A$. 
 Using this together with (\ref{eq:estimate-norm-A}), we obtain
$\|M^n(A-\underline{c})\|\lesssim\delta_0^n\|S\|_*$,
where the implicit constant is independent of $n$ and $S$. Therefore, for every
$Q\in \mathscr{D}_{k-p+1}^0$ we have
\begin{align*}
|\mathscr{U}_{S_n}(Q)-\mathscr{U}_{\T_{\underline{c}}}(Q)|
&\leq
\sum_{l=n}^{\infty}
\left|
\left(\mathscr{U}\circ (f^l)_*(Q)\right)M^{-l-1}\underline{c}
\right|
+d_p^{-n}\left|\mathscr{U}_{S}\big((f^{n})_*(Q)\big)\right| \\
&\quad +
\sum_{l=0}^{n-1}
\left|
\left(\mathscr{U}\circ (f^l)_*(Q)\right)
\big[{d_p^{-n}M^{n-l-1}(A-\underline{c})}\big]
\right| \\
&\lesssim
\sum_{l=n}^{\infty}
\left(\frac{\delta_{0}}{d_p}\right)^l
 \|S\|_*\|Q\|^{\bigstar}
+
\left(\frac{\delta}{d_p}\right)^n
M_S\|Q\|^{\bigstar}  +
\sum_{l=0}^{n-1}
\left(\frac{\delta_{0}}{d_p}\right)^n
{\|S\|_*}\|Q\|^{\bigstar}.
\end{align*}
Since ${\delta_0<}\delta<d_p$, 
it follows that
$
|\mathscr{U}_{S_n}(Q)-\mathscr{U}_{\T_{\underline{c}}}(Q)|
\lesssim
\left(\frac{\delta}{d_p}\right)^n
{(\|S\|_*+}M_S{ )}\|Q\|^{\bigstar},
$
where the implicit constant is independent of $S$, $Q$, and $n$.
This completes the proof.
 \end{proof}

  We can now conclude the proof of Proposition
 \ref{prop: Sp-convergence}.
\begin{proof}[Proof of Proposition \ref{prop: Sp-convergence}]
Let $\widetilde S\in \mathscr D_p$ be the standard 
form associated with $S$ given by Lemma
\ref{lem:decom_S=alpha+S'}
applied with $q=p$.
Since $\widetilde S$ is smooth, its super-potential $\mathscr U_{\widetilde S}$ is continuous with
respect to the $\bigstar$-topology. Lemma \ref{lem:decom_S=alpha+S'} also gives
\[
\{\widetilde S\}=\{S\}, \quad {\|\widetilde{S}\|_*\lesssim\|S\|_*\leq 1}
\qquad\text{and}\qquad
M_{\widetilde S}:={\sup_{\|\Omega\|^\bigstar\le 1}}|\mathscr U_{\widetilde S}(\Omega)|
\lesssim \|S\|_* \le 1.
\]
Hence, all the assumptions 
on $S$
in Proposition
\ref{prop: Sp-convergence-1} 
are satisfied by $\widetilde S$.
Therefore,
the sequence
$\widetilde S_n:=d_p^{-n}(f^n)^*(\widetilde S)$
converges $SP^{\bigstar}$-uniformly to a current $\T_{\underline c}\in \mathscr D_p$ depending only on the class $\underline c$.
Moreover, for every $\delta_f<\delta<d_p$,
by Proposition \ref{prop: Sp-convergence-1} 
we also  have
\[
|\mathscr U_{\widetilde S_n}(Q)-\mathscr U_{\T_{\underline c}}(Q)|
\lesssim
\left(\frac{\delta}{d_p}\right)^n{(\|\widetilde S\|_*+}
M_{\widetilde S}{)}\,\|Q\|^\bigstar
\lesssim
\left(\frac{\delta}{d_p}\right)^n \|Q\|^\bigstar
\]
for every $Q\in \mathscr D_{k-p+1}^0$,
where the implicit constants are independent of $S$, $Q$, and $n$.
This proves the assertion.
\end{proof}

\subsection{Proof of Theorem \ref{thm:new-speed}} \label{subsection:4.3} By homogeneity, we may assume without loss of generality that
\[
\|S\|_* \le 1, \qquad \|R\| \le 1, \qquad 
\|dd^cR\|^\bigstar
\le 1,
\qquad 
\text{and }
\quad\|\mathscr U_{dd^cR}\|_{\gamma,*} \le 1.
\]
Recall that we normalize
potentials and super-potentials 
as in Section \ref{subsection:superpotentials}.
The current $R$  of bidegree $(k-p,k-p)$
is a potential of $Q:=dd^cR\in\mathscr D_{k-p+1}^0$, but it is not
necessarily $\alpha_{ p}$-normalized.
Let $\mathscr U_{S_n}$ and $\mathscr U_+$ denote the $\alpha_{ p}$-normalized super-potentials of
$S_n:=d_p^{-n}(f^n)^*(S)$ and $\T^+$,
respectively. Set
\begin{equation}
\label{eq:Q-and-R'}
R' := R-\sum_{j=1}^{h( p)} c_j\,{ \check\alpha_{p,j}},
\qquad\text{where}\qquad
c_j:=\langle R,\alpha_{p,j}\rangle.
\end{equation}
Then $R'$ is an $\alpha_{ p}$-normalized potential of $Q$.
We then
have
 \begin{align}
\notag\bigg|\bigg\langle\frac{\left(f^{n}\right)^{*}(S)}{d_p^n}
  -c \T^+,  R\bigg\rangle\bigg|
 &\leq
 \left|\bigg\langle\frac{\left(f^{n}\right)^{*}(S)}{d_p^n}-c \T^+,R'\bigg\rangle\right|+
\left|\bigg\langle\frac{\left(f^{n}\right)^{*}(S)}{d_p^n}-c \T^+,R-R'\bigg\rangle\right|\\
&\label{eq:thm1-1} \leq 
 \bigg|\mathscr{U}_{{S_n}}(Q)-c \mathscr{U}_{+}(Q)\bigg|
 +\sum_{j}|c_j|\, \left|\bigg\langle\frac{\left(f^{n}\right)^{*}(S-c\T^+)}{d_p^n},{\check\alpha_{p,j}}\bigg\rangle\right|,
\end{align}
where in the last step we used the fact
that $R'$ is a normalized potential 
(for the first term)
and
the definition of $R'$ and the invariance of $\T^+$
(for the second term).
We now consider separately the two terms
in the last line.

\smallskip
\noindent\textbf{Step 1. Estimate of the first term.}
We have
\begin{equation}\label{e-thm13}
\bigg|\mathscr{U}_{{S_n}}(Q)-c \mathscr{U}_{+}(Q)\bigg|\lesssim
(\delta / d_p)^{n},
\end{equation}
where the implicit constant is independent of $S,Q$, and $n$. 

\begin{proof}[Proof of Step 1]

Let $\widetilde S\in\mathscr D_p$ 
be the standard 
form associated with $S$
as in Lemma \ref{lem:decom_S=alpha+S'}
 applied with $q=p$,
and set
$S^0:=S-\widetilde S\in \mathscr D^0_p$ 
as in that lemma. 
For every $n\in \mathbb N$, we 
have
\[
S_n=\widetilde S_n+S_n^0,
\qquad\text{where}\qquad
\widetilde S_n:=d_p^{-n}(f^n)^*(\widetilde S),
\quad
S_n^0:=d_p^{-n}(f^n)^*(S^0).
\]
By the linearity of super-potentials,
we have
\begin{align*}
    \bigg|\mathscr{U}_{{S_n}}(Q)-c \mathscr{U}_{+}(Q)\bigg|\leq\bigg|\mathscr{U}_{{\widetilde S_{n}}}(Q)-c \mathscr{U}_{+}(Q)\bigg|+\bigg|\mathscr{U}_{{S^0_n}}(Q)\bigg|.
\end{align*}
We now treat separately the first term (the smooth part) and the second term (the exact part) of the right-hand side of the above expression.

\smallskip
\noindent\textbf{Step 1a. The smooth part.}
Since $\widetilde S_n$ is smooth, $\mathscr U_{\widetilde S_n}$ is continuous with respect to the
$\bigstar$-topology.
{Proposition~}\ref{prop: Sp-convergence}
implies that the convergence
$\widetilde S_n\to \T_{\underline c}=c\T^+$ is $SP^{\bigstar}$-uniform
and that
\begin{align}\label{eq-est-for-beta-S}
\bigg|\mathscr{U}_{{\widetilde S_{n}}}(Q)-c \mathscr{U}_{+}(Q)\bigg|\lesssim \left(\frac{\delta}{d_p}\right)^n
\end{align}
 where the implicit constant is independent of $S$, $Q$,   and $n$. 

\smallskip
\noindent\textbf{Step 1b. The exact part.}
We now estimate $|\mathscr U_{S_n^0}(Q)|$.
Since $\{S^0\}=0$, we also have $\{S_n^0\}=0$.
Fix $\widetilde\delta$ such that $\delta_f<\widetilde \delta<\delta$, and define
$Q_n:=C^{-1}\widetilde\delta^{-n}(f^n)_*(Q)$,
where $C\ge 1$ is chosen large enough so that $\|Q_n\|^\bigstar\le 1$ for every $n$,
see Lemma 
\ref{lem:growth} {applied to $f^{-1}$
in bidegree $k-p+1$, noting that
$d_{k-p+1}(f^{-1})=d_{p-1}(f)<\widetilde\delta$.}
Applying Lemma
\ref{lem:Holder_constant}
with $s=p$, and using the assumption
$\|\mathscr U_Q\|_{\gamma,*}\le 1$, 
we get
\[
\|\mathscr U_{Q_n}\|_{\gamma,*}
=
C^{-1}\widetilde\delta^{-n}\|\mathscr U_{(f^n)_*(Q)}\|_{\gamma,*}
\le
C^{-1}
\widetilde\delta^{-n}A^{n\gamma}(c\,d_p^n)^{1-\gamma}
\lesssim
A^{n\gamma}
(d_p /\widetilde \delta)^n.
\]
 Since $\|S^0\|_*\lesssim 1$
(by Lemma~\ref{lem:decom_S=alpha+S'}(2)),
{ we may assume that $\|S^0\|_*>0$, since otherwise $S^0=0$
and the estimate is trivial.}
Proposition~\ref{prop:U_S-estimates} then gives 
\begin{align*} \big|\mathscr{U}_{S^0}({ Q}_n)\big|=\|S^0\|_*\,\big|\mathscr{U}_{S^0 /\|S^0\|_*}({ Q}_n)\big|\lesssim
1+\gamma^{-1}\log^+\|\mathscr{U}_{{ Q}_n}\|_{\gamma,*}\lesssim {(1+}n/\gamma{)},
\end{align*}
where the implicit 
constants are
independent of $S$, $Q$, and $n$. 
Finally, we deduce 
\begin{equation}\label{eq:thm1-4}
\left|\mathscr{U}_{S_n^{0}}(Q)\right|=C\left(\widetilde \delta / d_p\right)^{n}\left|\mathscr{U}_{S^{0}}\left({Q}_{n}\right)\right| \lesssim \left(\widetilde\delta /d_p\right)^{n}{ (1+n/\gamma)}\lesssim \left(\delta/ d_p\right)^{n},
\end{equation}
where again the implicit constant
is 
independent of  $S$, $Q$, and $n$.

 \medskip
 
 Combining (\ref{eq-est-for-beta-S}) and (\ref{eq:thm1-4})
 we obtain  (\ref{e-thm13}),
 which completes the proof of Step 1.
\end{proof}

\smallskip
\noindent\textbf{Step 2. Estimate of the second term.}
We have
\[
\sum_j |c_j|
\left|
\left\langle
\frac{(f^n)^*
(S-c\T^+)}{d_p^n},\,{\check\alpha_{p,j}}
\right\rangle
\right|
\lesssim
\left(\frac{\delta}{d_p}\right)^n,
\]
where the implicit constant is independent of $S$ and $n$. 

\begin{proof}[Proof of Step 2.]
Recall that
every
$\check\alpha_{p,j}$ is smooth, and
that $d_p^{-n}(f^n)^*(S-c\T^+)\to0$. Applying 
\cite[Corollary 3.3]{BD}
(see also \cite[Lemma 4.1]{Vergamini25}),
for each $j$
we obtain 
\begin{align*}
\left|\bigg\langle\frac{\left(f^{n}\right)^{*}(S-c\T^+)}{d_p^n},{\check\alpha_{p,j}}\bigg\rangle\right|
\lesssim
\left(\frac{\delta}{d_p}\right)^{n}\|S-c\T^+\|_*\,\|{\check\alpha_{p,j}}\|,
\end{align*}
where the implicit constant
is 
independent of
$S$, $c$, $j$, and $n$. 
Since the constant $c=c(S)$ depends linearly on 
the cohomology class $\{S\}$,
we have
\[\|S-c\T^+\|_*\leq \|S\|_*+|c|\,\|\T^+\|_*\lesssim \|S\|_*,\] 
where again the implicit constant 
is independent of  $S$, $c$ and  $\T^+$.  
We then obtain 
\begin{equation*}
\label{e-thm2-1}
\sum_{j}|c_j|\, \left|\bigg\langle\frac{\left(f^{n}\right)^{*}(S-c\T^+)}{d_p^n},{\check\alpha_{p,j}}\bigg\rangle\right|\lesssim \left(\frac{\delta}{d_p}\right)^n\, \|S\|_*\,\sum_{j}|c_j|
\lesssim
\left(\frac{\delta}  {d_p}\right)^n\, \|S\|_*,
\end{equation*}
where in the last step we used the definition 
(\ref{eq:Q-and-R'}) of the $c_j$'s
and the assumption $\|R\|\leq 1$ to get 
$\sum_j |c_j| \lesssim \|R\| \le 1$.
The desired estimate follows.
\end{proof}

The assertion follows
by
combining 
 (\ref{eq:thm1-1}) with the estimates in Steps 1 and 2.   The proof is complete.
 Note that the constant $B>0$ in Theorem \ref{thm:new-speed} depends on $\gamma>0$ as $B_\gamma\leq C/\gamma$ for some constant $C>0$ independent of $S, R$ and $n$. 
 
\section{Decay estimates outside the support of \texorpdfstring{$\T^+$}{}}
\label{sec:decay_estimates}
Throughout this section, $(X,\omega)$ denotes a compact K\"ahler manifold of dimension $k$, and $f\colon X\to X$ is a holomorphic automorphism with simple action on cohomology. 
We assume that $d_p$ is the main dynamical degree and fix an auxiliary dynamical degree $\delta<d_p$.
All constants appearing below may depend on $X$, $f$,
$\omega$, and 
$\delta$,
 but are independent of the integers $n_j$
 unless otherwise specified.

\subsection{Notations and a preliminary estimate}

For integers $n_1\ge\cdots\ge n_{ \ell}\ge0$ and a smooth $(j,j)$-form $\Theta$,
set
\begin{equation}\label{eq:def-omega}
\Omega_{n_1,\dots,n_{ \ell}}(\Theta)
:=
(f^{n_1})_*(\Theta)
\wedge
(f^{n_2})_*(\omega)
\wedge\cdots\wedge
(f^{n_{\ell}})_*(\omega),
\end{equation}
and
\begin{equation}\label{eq:def-omegaprime}
\Omega'_{n_1,\dots,n_{\ell}}(\Theta)
:=
\Omega_{n_1,\dots,n_{\ell}}(\Theta)\wedge\omega .
\end{equation}

\begin{proposition}\label{prop:M_q_general}
Let $j, \ell\ge1$ satisfy $ \ell+j-1\le k-p$.
Let $\Theta$ be a smooth $dd^c$-exact $(j,j)$-form 
with $\|\mathscr U_{\Theta}\|_{\gamma, *}\leq 1$  
for some $0<\gamma\leq 1$.
There exist constants $A>1$ and $C_{\gamma,{\ell}}>1$
such that for all $n_1\ge\cdots\ge n_{\ell}\ge0$
we have
\[
\|
\mathscr U_{\Omega'_{n_1,\dots,n_{ \ell}}(\Theta)}
\|_{\gamma,{*}}\le
C_{\gamma,{\ell}} A^{n_1\gamma}\delta^{n_1}.
\]
\end{proposition}

\begin{proof}
Define recursively
\[
R_1=(f^{n_1-n_2})_*\Theta
\quad
\text{and}
\quad
R_s=(f^{n_s-n_{s+1}})_*(\omega\wedge R_{s-1})
\quad (2\le s\le {\ell}),
\]
where we also set $n_{{ \ell}+1}:=0$ for convenience.
Then $R_{\ell}=\Omega_{n_1,\dots,n_{ \ell}}(\Theta)$.

Applying Lemma~\ref{lem:Holder_constant} and
Lemma~\ref{lem:product_superpotential} inductively,
we see that $\mathscr U_{R_{ \ell}}$ is $(\widetilde M_{ \ell}, \gamma)$-H\"older-continuous, where
\[
\widetilde
{M}_{ \ell}
=
 A_1^{({\ell}-1)\gamma} A^{n_1\gamma}
c^{(1-\gamma){\ell}}
\left(
\prod_{s=1}^{{\ell}}
\delta_{k-j-s+2}^{\,n_s-n_{s+1}}
\right)^{1-\gamma}.
\]
Since ${\ell}+j-1\le k-p$, the indices in the product belong to $\{p+1, \dots, k\}$.
Hence
we have
$\delta_{k-j-s+2}<\delta$ for all $s$, and
$\prod_{s=1}^{{\ell}}
\delta_{k-j-s+2}^{\,n_s-n_{s+1}}
\le
\delta^{\,n_1}$.
Thus, we have
\[
\widetilde
{M}_{\ell}
\le
 A_1^{({\ell}-1)\gamma} A^{n_1\gamma}
c^{(1-\gamma){\ell}}
\delta^{n_1}.
\]

Finally, since
$\Omega'_{n_1,\dots,n_{\ell}}(\Theta)=\omega\wedge R_{\ell}$,
a further application of Lemma~\ref{lem:product_superpotential}
gives the desired estimate.
\end{proof}

\subsection{Decay estimates outside the support of \texorpdfstring{$\T^+$}{}}

  Recall that $\T^+$ is the Green current of $f$, and that $d_p$ and $\delta$ are 
  the main and an auxiliary dynamical degrees.
Set $q:= k - p$ and
let $\chi$ be a smooth cut-off function supported outside $\supp \T^+$.
We also fix $0<\gamma \leq 1$.
The goal of this section is to prove the following proposition, giving a bound for the
  $\bigstar$-norm of
\[\Omega'_{n_1,\dots, n_q}(dd^c\chi)=\Omega_{n_1,\dots, n_q,0}(dd^c\chi)=\Omega_{n_1,\dots, n_q}(dd^c\chi)\wedge\omega.
\]
Unless specified,
the constant appearing below will depend on $\chi$ and $\gamma$, but not on the integers
$n_1,\dots, n_q$.

\begin{proposition}
    \label{pro:*-mass(k-p,k-p)} 
There exist constants $A>1$ (independent of $\gamma$)
and
$C:=C(\chi, \gamma)>0$
  such that
\[
\bigg\Vert\Omega'_{n_1,\dots,n_q}(dd^c\chi)\bigg\Vert^{\bigstar}\leq
C A^{n_1 \gamma}\delta^{n_1}
\]  
for all integers
$n_1\geq n_2\geq\dots \geq n_q\geq 0$. 
\end{proposition}

To prove Proposition~\ref{pro:*-mass(k-p,k-p)},
we will first introduce some auxiliary currents.
By Lemma~\ref{lem:phi_n^+}, we can write
$dd^c\chi=\beta_1^+-\beta_1^-$,
where 
$\beta_1^\pm$ are smooth positive forms with
$\supp\beta_1^\pm\cap\supp \T^+=\emptyset$.
For each $1\le s<q$ and $I\in\{+,-\}^s$,
apply Lemma \ref{lem:phi_n^+}
to $\beta_s^I$, obtaining
smooth positive forms $\beta_{s+1}^{I,+}$ and
$\beta_{s+1}^{I,-}$ such that
\[
dd^c\beta_s^I
  =\beta_{s+1}^{I,+}-\beta_{s+1}^{I,-}
\qquad
\text{and}
\qquad
\operatorname{supp}\beta_{s+1}^{I,\pm}
  \cap\operatorname{supp}\T^+=\varnothing.
\]
Thus, for every $1\le s\le q$, we have
$2^s$ smooth positive forms $\beta_s^I$,
indexed by $I\in\{+,-\}^s$, whose supports
are disjoint from $\operatorname{supp}\T^+$.
For $1\le s\le q$ and $|I|=q+1-s$, define
\[
R_s^I:=
(f^{n_1-n_s})_*(\beta_{q+1-s}^I)
\wedge
\bigwedge_{i=2}^s
(f^{n_i-n_s})_*(\omega),
\]
where an empty wedge product is understood
to be $1$. In particular, $R_1^I=\beta_q^I$.
For $2\le s\le q$, we also have
\[
R_s^I=
\Omega'_{n_1-n_s,\ldots,n_{s-1}-n_s}
(\beta_{q+1-s}^I).
\]
Moreover, all the $R_s^I$ are positive and satisfy
$\operatorname{supp}R_s^I
\cap\operatorname{supp}\T^+=\varnothing$,
by the invariance of $\operatorname{supp}\T^+$.
We also set
\begin{equation}
\label{eq:constantM}
M_s := \max_{|I|=q+1-s}
\|\mathscr U_{dd^c\beta_{q+1-s}^I}\|_{\gamma,{*}}
\quad
\text{ and }
\quad
M:= \max_s M_s.
\end{equation}

\begin{lemma}
\label{lem:A_first_reduction}
There exists a constant
$B>0$ such that
\[
\big\|
\Omega'_{n_1,\dots,n_q}(dd^c\chi)
\big\|^{\bigstar}
\le
B\,\delta^{n_q}
\sum_{\circ \in \{+,-\}}
\Big(
\|dd^c R_q^\circ\|^{\bigstar}
+
\|\mathscr U_{dd^c R_q^\circ}\|_{\gamma,{*}}
+ \|R_q^\circ\|
\Big).
\]
\end{lemma}

\begin{proof}
 Let $\mathscr{H}_{k-q-1}$
  denote the set of smooth positive closed 
  { $(k-q-1,k-q-1)$-}forms $\Phi$ with $\|\Phi\|\le1$.
By 
the definition of the $\bigstar$-norm 
and of the forms $\beta^{\pm}_1$ as above,
we have
\[
\big\|\Omega'_{n_1,\dots,n_q}(dd^c\chi)\big\|^{\bigstar}
\le
\sup_{\Phi\in\mathscr{H}_{k-q-1}}
\big\langle
\Omega'_{n_1,\dots,n_q}(\beta_1^+),\Phi
\big\rangle 
+
\sup_{\Phi\in\mathscr{H}_{k-q-1}}
\big\langle
\Omega'_{n_1,\dots,n_q}(\beta_1^-),\Phi
\big\rangle.
\]
Since the two terms can be treated in a similar way, it suffices to estimate the first
term.
Fix $\Phi\in \mathscr{H}_{k-q-1}$.
Using the definitions 
(\ref{eq:def-omega}) and (\ref{eq:def-omegaprime}),
 the corresponding 
component of the first
 term is equal to
\[
\big\langle
\Omega_{n_1,\dots,n_q}(\beta_1^+),\Phi\wedge\omega
\big\rangle
=
\big\langle
(f^{n_q})^*(\Phi\wedge\omega),R_q^+
\big\rangle.
\]
Since $\supp R_q^+\cap\supp \T^+=\emptyset$, Theorem~\ref{thm:new-speed}
applied with $S:=\Phi\wedge\omega$ and $R:=R_q^+$
gives
a constant $B>0$ (depending only on $\gamma$)
such that
\[
\big|\big\langle
(f^{n_q})^*(\Phi\wedge\omega),R_q^+
\big\rangle\big|
\le
B\,\delta^{n_q}\|\Phi\wedge\omega\|_*
\Big(
\|dd^c R_q^+\|^{\bigstar}
+
\|\mathscr U_{dd^c R_q^+}\|_{\gamma,{*}}
+ \|R_q^+\|
\Big),
\]
 as the term $\langle R^+_q, \T^+\rangle$
vanishes.
For $\Phi\in\mathscr{H}_{k-q-1}$
 we have $\|\Phi\wedge\omega\|_*=\|\Phi\wedge\omega\|=\|\Phi\|\le1$.
Taking the supremum over $\Phi\in\mathscr{H}_{k-q-1}$,
we obtain
\[
\sup_{\Phi\in\mathcal {H}_{k-q-1}}
\big\langle
\Omega_{n_1,\dots,n_q}(\beta_1^+),\Phi\wedge\omega
\big\rangle
\le
B\,\delta^{n_q}
\Big(
\|dd^c R_q^+\|^{\bigstar}
+
\|\mathscr U_{dd^c R_q^+}\|_{\gamma,{*}}+{\|R_q^+\|}
\Big).
\]
A similar 
estimate holds with $\beta_1^-$ in place of $\beta_1^+$, and the 
assertion follows.
\end{proof}

\begin{lemma}
\label{lem:B_recursive_star}
 We have
\[
\|dd^c R_q^{\pm}\|^{\bigstar}
\le
B^{q-1}\delta^{n_1-n_q}
\sum_{|I|=q-1}
\|dd^c R_1^{\pm,I}\|^{\bigstar}
+
\sum_{s=1}^{q-1}
B^{q-s}\delta^{n_s-n_q}
\sum_{|I|=q-s}
\Big(\|\mathscr U_{dd^c R_s^{\pm, I}}\|_{\gamma,{*}}
+ \|R_s^{\pm,I}\|
\Big),
\]
where $B$ is as in Lemma \ref{lem:A_first_reduction}.
\end{lemma}

\begin{proof}
When $q=1$, the estimate is immediate, since the right-hand side is simply $\|dd^c R_1^\pm\|^{\bigstar}$.
 We now assume that \(q\geq 2\).
 It is enough to 
 prove that, for every $1\le s\le q-1$
 and $I$ with $|I|=q-s$,
 we have
\begin{equation}\label{eq:one_step_star}
\|dd^c R_{s+1}^{I}\|^{\bigstar}
\le
B\,\delta^{\,n_s-n_{s+1}}
\sum_{\circ \in \{+,-\}}
\Big(
\|dd^c R_s^{I,\circ}\|^{\bigstar}
+
\|\mathscr U_{dd^c R_s^{I,\circ}}\|_{\gamma,{*}}
+
\|R_s^{I,\circ}\|
\Big),
\end{equation}
where  $B$ is as in Lemma~\ref{lem:A_first_reduction}.
The assertion then follows by developing this recursion formula.

\medskip

By definition of the $\bigstar$-norm and since $dd^c R_{s+1}^{I}$ is supported outside $\supp T^+$, we can argue as in Lemma~\ref{lem:A_first_reduction}:
for  $\Phi\in\mathscr{H}_{k-q-1}$
 we 
 have
\begin{align*}
 \|dd^c R_{s+1}^{I}\|^\bigstar\leq
 \sup_{\Phi\in\mathcal {H}_{k-q-1}}\big\langle (f^{\,n_s-n_{s+1}})^*(\Phi\wedge\omega),\,R_s^{I,+}\big\rangle +\sup_{\Phi\in\mathcal {H}_{k-q-1}}\big\langle (f^{\,n_s-n_{s+1}})^*(\Phi\wedge\omega),\,R_s^{I,-}\big\rangle.
\end{align*}
 Both terms on
 the right-hand side of the above expression
can be estimated in a similar
way; hence, we will bound the first one. For this,
we apply
Theorem~\ref{thm:new-speed}
with
$S:=\Phi\wedge\omega$ and $R:=R_s^{I,+}$, 
using that
$\|\Phi\wedge\omega\|_*\le1$.
Taking the supremum over $\Phi\in\mathscr{H}_{k-q-1}$ yields (\ref{eq:one_step_star})
and completes the proof.
\end{proof}

\begin{lemma}
\label{lem:C_holder_control}
There exist constants $A>1$ (independent of $\gamma$)
 and $C'>0$
such that for all $1\le s\le q$
{and $I$ with $|I|=q+1-s$}
we have
\[
\|\mathscr U_{dd^c R_s^{I}}\|_{\gamma,{*}}
\le
C' \, A^{n_1\gamma}\,\delta^{n_1-n_s}.
\]
\end{lemma}

\begin{proof}
 For $2\le s\le q$,
by construction, we have
\[
dd^c R_s^{I}
=
\Omega'_{n_1-n_s,\dots,n_{s-1}-n_s}\!\big(dd^c\beta_{q+1-s}^{I}\big).
\]
 When 
$s=1$
{(and so $|I|=q$),}
we have $R_1^{I}=\beta_q^{I}$ which is smooth   and by (\ref{eq:constantM})  we obtain $\|\mathscr{U}_{dd^c R_1^{I}}\|_{\gamma,*}\leq M$.
 If \(M_s=0\), the estimate is immediate. Otherwise, for
 $2\leq s\leq q$ we
apply Proposition~\ref{prop:M_q_general} with 
$\ell := s-1$,
$j:= q+2-s$,
and
$\Theta := dd^c\beta_{q+1-s}^{I}$,
 and then multiply the resulting estimate by \(M_s\).
It follows that there exist constants $A>1$
(independent of $\gamma$)
 and $C_{s-1}>1$ such that
$\mathscr U_{dd^c R_s^{I}}$ is $( \widetilde M_{s-1},\gamma)$-H\"older continuous with
\[
{\widetilde M_{s-1}}
\le
{M_s}\,C_{s-1}\,A^{(n_1-n_s)\gamma}\,\delta^{\,n_1-n_s}
\le
M\,C_{s-1}\,A^{n_1\gamma}\,\delta^{\,n_1-n_s},
\]
where {$M_s$ and $M$ are}
as in (\ref{eq:constantM}).
In particular,
\[
\|\mathscr U_{dd^c R_s^{I}}\|_{\gamma,{*}}
\le
M\,C_{s-1}\,A^{n_1\gamma}\,\delta^{\,n_1-n_s}.
\]
Setting $C':={ M}\cdot\max\{1, C_{1},\dots,C_{q-1}\}$ gives the assertion. 
\end{proof}

\begin{lemma}\label{lem:D-mass}
There exists a constant $C''$ such that 
for all $1\leq s \leq q$ and $I$
with $|I|=q+1-s$
we have
\[
\|R_s^I\|\leq  C'' \delta^{n_1-n_s}.
\]
\end{lemma}

\begin{proof}
When $s=1$, we have $R_1^I=\beta_q^I$
with $|I|=q$, and the estimate is immediate.
By definition, we have
\[0\leq R_s^I
=
(f^{\,n_1-n_s})_*(\beta_{q+1-s}^{I})
\wedge
\bigwedge_{i=2}^s
(f^{\,n_i-n_s})_*(\omega)\lesssim
(f^{\,n_1-n_s})_*(\omega^{q+1-s})
\wedge
\bigwedge_{i=2}^s
(f^{\,n_i-n_s})_*(\omega),
\]
where the implicit constant depends only on $s$ and 
 the choice of the $\beta_s^I$'s
(in particular, it can be taken independent of $s$ and $I$). Observe that the last expression 
can be written as  the wedge product of $q =k-p$ 
positive closed $(1,1)$-currents, and that 
the one 
corresponding to $i=s$
is equal to $\omega$. 
If $q=1$ then $s=1$ and  the estimate is immediate
since the wedge product reduces to $\omega$. If $q\geq 2$, we can then apply Lemma
\ref{lem:pull-back-mass}
(with $f^{-1}$ instead of $f$) 
to the current
$(f^{\,n_1-n_s})_*(\omega^{q+1-s})
\wedge
\bigwedge_{i=2}^{s-1}
(f^{\,n_i-n_s})_*(\omega)$,
 which gives
$\|R_s^I\|\lesssim \delta^{n_1-n_s}$.
The assertion follows.
\end{proof}

\begin{proof}[Proof of Proposition~\ref{pro:*-mass(k-p,k-p)}]
Combining Lemmas~\ref{lem:A_first_reduction},
\ref{lem:B_recursive_star},
\ref{lem:C_holder_control}, and 
\ref{lem:D-mass}
we obtain
\[
\big\|
\Omega'_{n_1,\dots,n_q}(dd^c\chi)
\big\|^{\bigstar}
\le
(2B)^{q}{ \delta^{n_1}}\Big(
{\max_{|I|=q}\|dd^c R_1^{I}\|^{\bigstar}}
+{q\, (C'+C'')\,A^{n_1\gamma}}\Big),
\]
where the constants 
 $A,B,C', C''$
 are as in those lemmas.
The assertion follows by setting $C:= (2B)^{q}\big(
{\max_{|I|=q}}\|dd^c R_1^{I}\|^{\bigstar}+{q\,(C'+C'')}\big)$
{and observing that
the forms $R_1^I=\beta_q^I$
are fixed and depend only on $\chi$.}
\end{proof}

\begin{remark}\label{r:dim2}
The
estimate in Proposition \ref{pro:*-mass(k-p,k-p)}
is the crucial reason for introducing the norm $\|\cdot\|^\bigstar$, 
and the need for
the generalization of the convergence speed of \cite{DS10} given by Theorem \ref{thm:new-speed} 
with respect to this norm.
 When $k=2$, the auxiliary positive currents
appearing in the localized estimates have top
bidegree and are therefore automatically
closed. In this case, the convergence theorem
of \cite{DS10} suffices, and the new norm
is not needed.
\end{remark}

\section{Supports of measures of large entropy}
\label{s:proof-thm}

\subsection{Proof of Theorem~\ref{thm: dimX=k, p}}
It is enough to prove that $\supp \nu \subseteq \supp \T^+$.  Indeed, 
since $f^{-1}$ also has simple action on cohomology,
the same arguments applied to $f^{-1}$  yield
$\supp\nu\subseteq \supp \T^-$,
and therefore $\supp\nu \subseteq \supp\T^+\cap\supp\T^-$.

Assume that $\nu$ gives positive mass to a compact set 
$K\subset X\setminus\supp \T^+$.
Let $W\Subset W_1$ be open neighborhoods of $K$ such that 
$\overline{W_1}\cap\supp \T^+=\emptyset$,
and let  \(0\le\chi\le1\)
be a smooth cut-off function
which is equal to $1$ on $W$
and is supported in $W_1$.
By a now-classical construction due to Gromov \cite{Gromov03}, we have
\begin{equation}\label{eq:entropy-gromov}
h_t(f,K)\le \lov(f,W)
:=\limsup_{m\to\infty}\frac1m\log\volume(\Gamma_m^W),
\end{equation}
where $h_t(f,K)$ denotes the
 topological entropy of $f$ on $K$,
\[
\Gamma_m^W=\{(z,f(z),\dots,f^{m-1}(z)):\ z\in W\}
\]
and  denoting the natural projection  $\pi_j:X^m\to X$, we have 
\[
k!\, \volume(\Gamma_m^W) =
 \int_{\Gamma_m^W} \big(\sum_{j=0}^{m-1}\pi_j^*\omega\big)^k=
\sum_{0\le m_i\le m-1}
\int_W
\bigwedge_{i=1}^k (f^{m_i})^*(\omega).
\]
Since the number of terms 
in the last sum
is $m^k$, it suffices to estimate each integral separately.
Recall that $\delta$ denotes an auxiliary dynamical degree for $f$.

\begin{proposition}\label{prop:single_term_estimate}
Fix $0<\gamma\le1$.
There exist
constants $A>1$ (independent of $\gamma$) and 
 $C =C(\gamma)>0$ 
  such that for all integers
$0\le m_k\le\cdots\le m_1\le m-1$ we have
\[
\int_W
\bigwedge_{i=1}^{k}(f^{m_i})^*(\omega)
\le
C_\gamma\,A^{m\gamma}\delta^{m}.
\]
\end{proposition}

Assuming Proposition~\ref{prop:single_term_estimate} for now, 
we deduce from (\ref{eq:entropy-gromov}) that
$\volume(\Gamma_m^W)\lesssim m^k A^{m\gamma}\delta^m$,
where the implicit constant depends on $\gamma$ but is independent of $m$.
Hence
$h_t(f,K)\le \log\delta+\gamma\log A$.
 Since $h_\nu(f)>\log\delta_f$, we may first choose
$\delta_f<\delta<d_p$, and then $\gamma>0$ small
enough that
$\log\delta+\gamma\log A<h_\nu(f)$.
On the other hand, since $\nu$ is ergodic and
$\nu(K)>0$, Bowen's entropy estimate for sets
of positive measure  \cite{Bowen73}
gives
$h_\nu(f)\le h_t(f,K)$,
a contradiction. Therefore, we have
$\operatorname{supp}\nu\subset\operatorname{supp}\T^+$,
which completes the proof.
\subsection{Proof of Proposition~\ref{prop:single_term_estimate}}
We may assume $0\le m_k\le\cdots\le m_1\le m-1$.
Since 
 $W\subset\supp\chi \Subset W_1$
  and $\chi\equiv1$ on $W$,
  we have
\[
\int_W\bigwedge_{i=1}^{k}(f^{m_i})^*(\omega)
\le
\Big\langle\bigwedge_{i=1}^{k}(f^{m_i})^*(\omega),\chi\Big\rangle 
=
\Big\langle
\bigwedge_{i=1}^{k}(f^{m_i-m_{p+1}})^*(\omega),
(f^{m_{p+1}})_*(\chi)
\Big\rangle.
\]

Set
\begin{equation}\label{eq:SR}
S:=\bigwedge_{i=1}^{p}(f^{m_i-m_p})^*(\omega),
\quad
\text{ and }
\quad
R:=(f^{m_{p+1}})_*(\chi)
\bigwedge_{i=p+1}^{k}(f^{m_{p+1}-m_i})_*(\omega).
\end{equation}
Then
\[
\int_W\bigwedge_{i=1}^{k}(f^{m_i})^*(\omega)
\leq
\Big\langle
(f^{m_p-m_{p+1}})^*(S),R
\Big\rangle.
\]
Since $\supp \T^+$ is $f$-invariant and
$\supp\chi\cap\supp \T^+=\emptyset$, we obtain
$\supp R\cap\supp \T^+=\emptyset$.
We will estimate the above expression by means of Theorem \ref{thm:new-speed}. 
For this,
we need to bound $\|S\|_*$, 
$\|dd^c R\|^{\bigstar}$, 
$\|\mathscr U_{dd^c R}\|_{\gamma, *}$, 
and $\|R\|$.
The implicit constants below can depend on $\gamma$ but are independent of the integers $m_i$.

\medskip

We first estimate $\|S\|_*$. 
If $p=1$, the estimate   is immediate since $S=\omega$.
Otherwise,
by Lemma~\ref{lem:pull-back-mass} (applied with $q=p-1$), we have
\[
\|S\|_* =\|S\|
\lesssim \delta^{m_1-m_p}.
\]
Next, using the notation
(\ref{eq:def-omega})–(\ref{eq:def-omegaprime}), we have
\[
dd^c R
=
\Omega'_{m_{p+1},\,m_{p+1}-m_{ k},\,\dots,\,m_{p+1}-m_{ p+2}}(dd^c\chi).
\] 
Proposition~\ref{pro:*-mass(k-p,k-p)}
(applied to the integers $(m_{p+1},m_{p+1}-m_{ k},...,m_{p+1}-m_{p+2})$ instead of $n_1,...,n_q$)
 gives
\[
\|dd^c R\|^{\bigstar}
\lesssim
A^{m_{p+1} \gamma}\delta^{m_{p+1}}.
\]

Set
$M_\gamma:=\max\{1,\|\mathcal U_{dd^c\chi}\|_{\gamma,*}\}$,
which is finite since $\chi$ is smooth. Applying
Proposition 
\ref{prop:M_q_general}
with $j:=1$, $\ell:=k-p$, and
$\Theta:=M_\gamma^{-1}dd^c\chi$, using the integers
$(m_{p+1},m_{p+1}-m_k,\ldots,m_{p+1}-m_{p+2})$,
and then multiplying by $M_\gamma$, we obtain
\[
\|\mathscr U_{dd^c R}\|_{\gamma,{ *}}
\lesssim
A^{m_{p+1}\gamma}\delta^{m_{p+1}}.
\]
{Finally,
we bound $\|R\|$.
From the definition of $R$
(i.e., since $0\leq \chi \leq 1$),
it is enough to bound the mass of 
$\bigwedge_{i=p+1}^{k}(f^{m_{p+1}-m_i})_*(\omega)$. 
Observe that this is the wedge product
of $k-p$ push-forwards of $\omega$,
and that 
the term corresponding to $i=p+1$
is equal to $\omega$. 
If $k-p=1$, the estimate is immediate
since the wedge product reduces to $\omega$.
Otherwise, 
we deduce from 
Lemma \ref{lem:pull-back-mass}
(applied
to the current
$\bigwedge_{i=p+2}^{k}(f^{m_{p+1}-m_i})_*(\omega)$
with $f^{-1}$ instead of $f$ and $k-p$ instead of $p$)
that
\[
\displaystyle \|R\|
\lesssim \delta^{m_{p+1}-m_k}.
\]}

Since $\supp R\cap\supp \T^+=\emptyset$,
 Theorem~\ref{thm:new-speed}
implies
 that
\[
\Big\langle
(f^{m_p-m_{p+1}})^*(S),R
\Big\rangle
\lesssim
\delta^{m_p-m_{p+1}}
\|S\|_*
\Big(
\|dd^cR\|^{\bigstar}
+
\|\mathscr U_{dd^cR}\|_{\gamma, *}
{+\|R\|}
\Big).
\]
Combining the previous estimates yields
\[
\int_W\bigwedge_{i=1}^{k}(f^{m_i})^*(\omega)
\lesssim
A^{m_{p+1}\gamma}\delta^{m_1}
\le
A^{m\gamma}\delta^m,
\]
which proves the proposition.

\section{Examples}
\label{section:example}
\subsection{Julia-type sets}
Let $f:X\to X$ be an automorphism as in
Theorem~\ref{thm: dimX=k, p}, with main Green
currents $\T^+$ and $\T^-$. Set
\[
\J_f^\pm:=\supp\T^\pm
\qquad
\text{and}
\qquad
\J_f:=\J_f^+\cap\J_f^-.
\]
Another natural Julia-type set is the complement
of the Fatou set. Denote by
$\operatorname{Fat}^\pm(f)$ the loci where the
positive and negative iterates, respectively,
form normal families, and set
\[
\J_*(f):=
X\setminus\bigl(
\operatorname{Fat}^+(f)\cap
\operatorname{Fat}^-(f)\bigr).
\]
Normality of the positive iterates implies
that the normalized pull-backs defining
$\T^+$ converge to zero locally.
Thus
$\J_f^+\cap\operatorname{Fat}^+(f)
=\varnothing$.
The analogous statement holds for $\T^-$
and the negative iterates. Consequently, we have
$ \J_f\subseteq\J_f^+\cup\J_f^-\subseteq\J_*(f)$.

\begin{proposition}\label{example:sharp-example2}
There exist a smooth projective fourfold $X$
and a holomorphic automorphism $F:X\to X$
with simple action on cohomology such that
$\J_F^+\cup\J_F^-\subsetneq\J_*(F)$.
Moreover, there exists an ergodic
$F$-invariant probability measure $\nu$
such that $h_\nu(F)=\log\delta_F$
and
$\supp\nu\cap(\J_F^+\cup\J_F^-)=\varnothing$.
In particular, the entropy bound in
Theorem~\ref{thm: dimX=k, p} is optimal.
\end{proposition}
\begin{proof}
\textbf{Step 1. The surface example.} 
By a construction of McMullen~\cite[Theorem~1.4]{McM07}, there exists a smooth projective rational surface \(S\), obtained by blowing up \(\mathbb P^2\) at finitely many points, and an automorphism
\(g:S\to S\) with positive entropy and with a
 two-dimensional
Siegel disk centered at a fixed point \(p\in S\). Let \(d>1\) be the main dynamical degree of \(g\).
By \cite{Cantat01}, all eigenvalues of the action of
\(g^*\) on \(H^{1,1}(S,\mathbb R)\) have modulus \(1\), except for the two simple eigenvalues \(d\) and \(1/d\).
Thus, in this case, the auxiliary dynamical degree is \(\delta_g=1\).
Let $T^+$ and $T^-$ be the Green $(1,1)$-currents
of $g$, satisfying
$g^*T^+=dT^+$ and $g_*T^-=dT^-$.
Set $J^\pm:=\supp T^\pm$.
We claim that the Siegel disk is disjoint
from both $J^+$ and $J^-$.
Indeed, by the definition of a Siegel disk, there is a \(g\)-invariant neighborhood \(U\) of \(p\) and a  biholomorphism
$\psi:U\to\psi(U)\subset\mathbb C^2\) such that
$\psi\circ g\circ\psi^{-1}(z_1,z_2)
=(\alpha z_1,\beta z_2)$,
where \(|\alpha|=|\beta|=1\) are multiplicatively independent
(i.e.,
\(\alpha^m\beta^n=1\) for \(m,n\in\mathbb Z\) implies \(m=n=0\)).
Hence, for every relatively compact set \(V\Subset U\), the norms
\(\|g^n\|_{\mathcal C^1(V)}\) are uniformly bounded for \(n\in\mathbb Z\).
It follows that $d^{-n}(g^n)^*\omega\to0$
on \(V\) as \(n\to+\infty\), where \(\omega\) is a K\"ahler form on \(S\).
On the other hand,
\(d^{-n}(g^n)^*\omega\to cT^+\) for some \(c>0\).
Since \(V\Subset U\) is arbitrary, we obtain
\(U\cap J^+=\emptyset\).
Applying the same argument to \(g^{-1}\), we also obtain
\(U\cap J^-=\emptyset\).
In particular, we have
 $p\notin J^+\cup J^-$.

\smallskip
\textbf{Step 2. The product automorphism.}
Set
\(X=S\times S\), which is a smooth projective fourfold, and consider the automorphism
\begin{equation*}
F:=g\times g:X\to X,\qquad
F(x,y)=(g(x),g(y)).
\end{equation*}
 For a rational surface $S$  we have $H^{s,r}(S,\mathbb{C})=0$ for $s\neq r,$ and the  K\"unneth formula \cite[Theorem 11.38]{Voisin}  gives a canonical
isomorphism
\begin{align}\label{example:1}
H^{2,2}(X,\mathbb{R})
\simeq\bigoplus_{0\leq r\leq 2}
\left(H^{r,r}(S,\mathbb{R})\otimes H^{2-r,2-r}(S,\mathbb{R})\right).
\end{align}
Therefore
$d_q(F)=\max_{i+j=q}d_id_j$,
and the dynamical degrees of \(F\) are
\begin{equation*}
(d_0(F),d_1(F),d_2(F),d_3(F),d_4(F))
=(1,d,  d^2,d,1).
\end{equation*}
Thus \(d_2(F)=d^2\) is the main dynamical degree of \(F\).
On the middle summand
\(H^{1,1}(S,\mathbb R)\otimes H^{1,1}(S,\mathbb R)\)
of (\ref{example:1}), the action of \(F^*\) is
\(g^*\otimes g^*\).
Its eigenvalues are pairwise products of the eigenvalues of $g^*$ on $H^{1,1}(S,\mathbb{R})$.
Since $d$ is  a   simple eigenvalue of \(g^*\), the eigenvalue \(d^2\) of \(F^*\) is also simple. Moreover, every other eigenvalue on \(H^{2,2}(X,\mathbb R)\) has modulus at most \(d\).
On the two remaining summands
\(H^{0,0}(S,\mathbb R)\otimes H^{2,2}(S,\mathbb R)\) and
\(H^{2,2}(S,\mathbb R)\otimes H^{0,0}(S,\mathbb R)\),
the map \(F^*\) acts with eigenvalue \(1\).
It follows that \(F\) has simple action on cohomology and $\delta_F=d$. 
Let $\pi_1,\pi_2:X\to S$ denote the natural
projections. Up to positive multiplicative
constants, the main Green currents of $F$ are
$\T_F^\pm=
\pi_1^*T^\pm\wedge\pi_2^*T^\pm$.
Consequently, we have
\[
\J_F^\pm:=\supp\T_F^\pm
=J^\pm\times J^\pm.
\]
\smallskip
\textbf{Step 3. { Strict inclusion of Julia-type sets.}}
Choose $x\in J^+$. By the normality argument
above, $x\notin\operatorname{Fat}^+(g)$.
Consider $z:=(x,p)\in X$.
Since $p\notin J^+\cup J^-$, the product
formulas imply that
$z\notin\J_F^+\cup\J_F^-$.
On the other hand, if the family
$\{F^n\}_{n\ge0}$ were normal near $z$,
restricting to the slice $y\mapsto(y,p)$
and composing with $\pi_1$ would imply
normality of $\{g^n\}_{n\ge0}$ near $x$,
a contradiction.
Hence, we have
$z\in
\J_*(F)\setminus(\J_F^+\cup\J_F^-)$,
which proves the desired strict inclusion.

\smallskip
\textbf{Step 4. { Sharpness of the entropy threshold.}}
Finally, set $\nu:=\mu_g\otimes \delta_p$, where \(p\) is the Siegel fixed point introduced in Step~1 
and $\mu_g$ is the
Green probability measure of $g$.
Since \(p\) is fixed and $\mu_g$ is ergodic, \(\nu\) is an ergodic $F$-invariant probability measure  and satisfies
$h_\nu(F)=h_{\mu_g}(g)=\log d=\log\delta_F$.
Moreover, since $p\notin J^+\cup J^-$,
we have
$\supp\nu
=\supp\mu_g\times\{p\}
\subseteq
X\setminus(\J_F^+\cup\J_F^-)$.
This completes the proof.
 \end{proof}

\subsection{The \texorpdfstring{$\bigstar$}{bigstar}-norm on non-closed currents}
\label{s:norm-infinity}

The following example shows that the expression
in~\eqref{def:bigstar_norm} may a priori be infinite even
on positive currents of finite mass.
Set $E=\mathbb C/(\mathbb Z+i\mathbb Z)$, endowed
with the form $\sigma=dx\wedge dy$, and set
$X=E\times E$. Denote by $\pi_1,\pi_2$ the
projections and set
$\alpha:=\pi_1^*\sigma$,
$\beta:=\pi_2^*\sigma$, and
$\omega:=({\alpha+\beta})/{\sqrt{2}}$.
Thus $\int_X\omega^2=1$.
Choose a non-negative function $u\in L^1(E)$
such that $u(z)=|z|^{-1}$ near the origin,
and consider the positive $(1,1)$-current
$S:=(u\circ\pi_1)\beta$.
Its mass is finite, since
\[
\|S\|=\frac{1}{\sqrt{2}}\int_E u\,\sigma
<+\infty.
\]

Let $\eta\in C_c^\infty(\mathbb D)$ be
non-negative, with $\int_{\mathbb D}\eta\,dA=1$.
For $j$ sufficiently large, define
$\rho_j(z):=j^2\eta(jz)$ in a coordinate
neighbourhood of the origin, extended by zero
to $E$, and set
$\Phi_j:=\sqrt{2}\,(\rho_j\circ\pi_1)\alpha$.
Since $\rho_j\sigma$ is a top-degree form on
$E$, each $\Phi_j$ is smooth, positive and closed.
Moreover, we have
$\|\Phi_j\|=\int_E\rho_j\,\sigma=1$,
so that $\Phi_j\in\mathscr H_1(X)$.

By the definition of the $\bigstar$-norm, we have
\[
\|S\|^{\bigstar}
\ge \langle S,\Phi_j\rangle
=\sqrt{2}\int_E u\rho_j\,\sigma
=\sqrt{2}\,j
  \int_{\mathbb D}\frac{\eta(w)}{|w|}\,dA(w)
  \longrightarrow+\infty.
\]
Hence $\|S\|^{\bigstar}=+\infty$,
although $S$ is positive and has finite mass.

\section{Final Remarks and further questions}
\label{s:final}
\subsection{One-sided localization}
Let $f$ be an automorphism with simple action
on cohomology, and let $p$ be its main index.
Denote by $\rho_p$ the maximum modulus of the
eigenvalues of $f^*$ on $H^{p,p}(X,\mathbb R)$
different from $d_p$, and set
\[
\lambda_-:=\max\{d_{p-1},\rho_p\}
\qquad
\text{and}
\qquad
\lambda_+:=\max\{d_{p+1},\rho_p\}.
\]
It is natural to ask whether the two conclusions
of Theorem~\ref{thm: dimX=k, p} hold separately,
with the corresponding entropy thresholds.
More precisely, is it true that, for every
ergodic invariant probability measure $\nu$, we have
\[
h_\nu(f)>\log\lambda_-
\,\Rightarrow\,
\supp\nu\subseteq\supp\T^-
\quad
\text{and}
\quad
h_\nu(f)>\log\lambda_+
\,\Rightarrow\,
\supp\nu\subseteq\supp\T^+?
\]
Our proof does not establish these statements,
since the localized estimates in
Section~\ref{s:proof-thm} involve both
$d_{p-1}$ and $d_{p+1}$
when proving each inclusion.
 One-sided spectral gaps already play a role in the
construction of Green currents for automorphisms
\cite{DS05}, as well as in their rigidity and
equidistribution properties; see
\cite[Theorems~8.13--8.14]{DS4}.
See also \cite[Theorem~1.1]{LV26} for a
generalization of the construction to
holomorphic correspondences.

\subsection{Non-invertible maps and correspondences}
The invertibility of $f$ is used at several points in our method, as it allows us to transfer
the action and the
estimates from pull-backs to push-forwards. For example, in
Lemma~\ref{lem:growth} we use the identity
$\langle (f^n)^*R,\Phi\rangle = \langle R,(f^n)_*\Phi\rangle$ and then control the mass of $(f^n)_*\Phi$.
Invertibility is also used in Lemma~\ref{lem:Holder_constant}. There, the control of the
$\mathcal C^{-2}$-norm of $(f^n)^*S$ is obtained by moving the action to test forms and
using the smoothness of $f^{-1}$ to estimate the $\mathcal C^2$-norm of $f_*\Phi$. This
step has no direct analogue for a non-invertible map in the form needed here.

For these reasons, our localization argument
does not directly extend to the
non-invertible case, nor more generally to holomorphic correspondences.
The lower-regularity
techniques and related tools developed in
\cite{Vergamini25,LV26} would likely be useful in this direction.
 We also refer to \cite{Ahn16,AV24} for further related equidistribution results, of which a localized version would be useful here.

\subsection{Further possible extensions}
 It is also natural to ask whether
Theorem~\ref{thm: dimX=k, p} extends 
to other settings, such as generic birational maps of $\mathbb P^k$
\cite{BedfordDiller2005,dTV, Diller1996}.
 In these settings, the presence of indeterminacy sets prevents a direct application of our argument, and  in particular the quantitative convergence estimates used here are not presently available in a form compatible with localized non-closed currents.
 
One can also ask the same question in
local or noncompact
settings such as H\'enon-like maps \cite{BDR,BDR2,DNS,DS2006},
where cohomological
tools are no longer available,
or 
attractors as in 
\cite{Dinh-attracting-current, Taflin18}. 
 In both cases, one has
invariant Green-type currents,  
but they are in general not built from Green $(1,1)$-currents.
One would need convergence estimates towards these
Green-type currents
which are
uniform in the integers $n_1\ge \cdots \ge n_\ell\ge 0$ for mixed terms of the form
\[
(f^{n_1})^*(\Theta)\wedge (f^{n_2})^*(\omega)\wedge \cdots \wedge (f^{n_\ell})^*(\omega),
\]
or for their localized analogues obtained by taking $\Theta=dd^c\chi$, and cohomology tools cannot be used to study inductively
the single factors.
An extension of our theorem would likely require a corresponding
 super-potential convergence theorem. 
We refer to
\cite{Dujardin2004,Dujardin2020}
for related results in dimension $2$, where the above difficulty is not present.

\appendix \section{Regularization of currents  and  solution of the  \texorpdfstring{$dd^c$}{}-equation
}

\subsection{Regularization of currents}\label{appendix:Regularization}
In this subsection we recall
the definition of the regularization operators
$\mathcal{L}_\theta:\mathscr{D}_s(X)\to \mathscr{D}_s(X)$, $0\le s\le k$,
introduced in \cite[Sections~2.3--2.4]{DS10JAG}.
They are used to control super-potentials along structural lines and will be the main tool in the proof of 
Proposition \ref{prop:U_S-estimates}.
Recall that $(X,\omega)$ is
a compact K\"ahler manifold of dimension $k$.
Let $\Delta\subset X\times X$ be the diagonal and 
$\pi:\widehat{X\times X}\to X\times X$
the blow-up along $\Delta$,
with exceptional divisor $\widehat\Delta:=\pi^{-1}(\Delta)$.
Denote by $\Pi_1,\Pi_2:\widehat{X\times X}\to X$ the compositions of $\pi$ with the two natural projections from $X\times X$ to its factors.
Let also
$\tau_0:\P^1\times\widehat{X\times X}\to \P^1$ and
$\tau:\P^1\times\widehat{X\times X}\to \widehat{X\times X}$ be the natural projections.

We use the notation of \cite[Example~2.3.1]{DS10JAG}.
Let $\eta$ be a smooth real closed $(k-1,k-1)$-form on $\widehat{X\times X}$ 
such that $\pi_*([\widehat{\Delta}]\wedge\eta)$  is cohomologous to $[\Delta]$.
Let $\gamma$ be a smooth real closed $(1,1)$-form cohomologous to $[\widehat\Delta]$.
{We choose $\eta$ and $\gamma$ invariant
under the involution of
$\widehat{X\times X}$ induced by
$(x,y)\mapsto(y,x)$, as in
\cite[Section 2.1]{DS10JAG}.}
Then there exists a quasi-p.s.h.\ function $\varphi$ on $\widehat{X\times X}$ such that
$dd^c\varphi=[\widehat\Delta]-\gamma$.
 Up to adding a constant, we assume
that $\varphi<-2$, as in
\cite[Section 2.1]{DS10JAG}.
Set $\Theta_0:=[\widehat\Delta]\wedge\eta$ and, for $\theta\in\P^1$,
\begin{equation}\label{eq:Spesific-theta}
\Theta_\theta := (dd^c\varphi_\theta+\gamma)\wedge\eta,
\end{equation}
where $\varphi_\theta$ 
{ is the regularization
of $\varphi$ constructed in
\cite[Section 2.4]{DS10JAG}, satisfying}
$\varphi_\theta\searrow \varphi$ as $|\theta|\to 0$.
For a current ${S}$ on $X$, define the transform
\[
\mathcal L_\theta(S):=(\Pi_2)_*\big(\Pi_1^*(S)\wedge\Theta_\theta\big)
\]
and for convenience denote by  $S_{\theta}:=\mathcal{L}_{\theta}(S)$. 
The family $(S_{\theta})_{\theta \in \mathbb P^1}$
is called the \emph{special structural line} associated to $S$.
The following proposition and lemma
summarize
the main properties of such transforms that we will need,
 see \cite[Sections 2.3 and 2.4]{DS10JAG}
 and \cite[Section 3.2]{Vergamini25}.

\begin{proposition}\label{prop:property-transform}
The transforms $\mathcal L_\theta$ satisfy the following properties:
\begin{enumerate}
\item $\mathcal L_\theta$ depends only on $|\theta|$
and $\mathcal L_0(S)=S$;
\item $\mathcal L_\theta=\mathcal L_\infty$ for $|\theta|\ge1$;
\item $\{\mathcal L_\theta(S)\}=\{S\}$ for every $\theta\in\P^1$ and $S\in\mathscr D_s(X)$;
\item $\mathcal L_\theta$ is continuous 
with respect to the $*$-topology for every $\theta\in\P^1$;
\item $\|\mathcal L_\theta(S)\|_{*}\le c\,\|S\|_{*}$, with $c>0$ independent of $\theta$ and $S$;
\item for $|\theta|\le1$ we have $\dist_{2}(\mathcal L_{\theta}(S), S) \le c|\theta|\,\|S\|_{*}$,
with $c>0$ independent of $\theta$ and $S$.
\end{enumerate}
\end{proposition}

\begin{lemma}
\label{lem:L-infty}
For $1 \leq m \leq k$, set $q_{m}=\frac{k+1}{k-m+1}$. Take $S \in \mathscr{D}_{s}^{0}$. There exists a constant $c>0$, independent of $S$, such that:
\begin{enumerate}
\item $\left\|\mathcal{L}_{\infty}(S)\right\|_{L^{q_{1}}} \leq c\|S\|_{*}$;
\item $\left\|\mathcal{L}_{\infty}(S)\right\|_{L^{q_{m+1}}} \leq c\|S\|_{L^{q_{m}}}$ for every $1 \leq m<k$;
\item $\left\|\mathcal{L}_{\infty}(S)\right\|_{L^{\infty}} \leq c\|S\|_{L^{q}}$ for every $k+1 \leq q \leq+\infty$.
\end{enumerate}
\end{lemma}

 The following result  was established 
in \cite[Proposition 3.2.4]{DS10JAG} with 
$\|{Q}\|_*$
instead of $\|{ Q}\|^{\bigstar}$.

\begin{proposition}\label{prop:star-bigstar-ddc}
Let $\left({ Q}_{\theta}\right)_{\theta \in \mathbb{P}^{1}}$ be the special structural line associated to a smooth form ${Q} \in \mathscr{D}_{k-{s}+1}^{0}$. 
Take $S\in \mathscr{D}_{{ s}}^0$. 
Then $\theta\mapsto \mathscr U_S({Q}_\theta)$ is a continuous
function on $\P^1$
which is constant on $\{|\theta|\ge1\}$.
Moreover, there exists $c>0$, independent of ${Q}$ and $S$, such that
\[
\big\|dd^c_\theta\,\mathscr U_S({Q}_\theta)\big\|_* \le c\,\|S\|_*\,\|{Q}\|^{\bigstar},
\]
where $dd^c_\theta$ is taken on $\P^1$
 and the norm $\|\cdot\|_*$
 on the left-hand side
  denotes the mass of the resulting signed measure on $\P^1$.
\end{proposition}
\begin{proof}
Since ${Q}$ is smooth, ${ Q}_{\theta}$ is smooth for every $\theta$, and by 
\cite[Lemma 2.4.6]{DS10JAG}
the function $H(\theta) := \mathscr{U}_{S}({ Q}_{\theta})$ is continuous on $\mathbb{P}^{1}$.
 It remains to bound the mass of $dd^c H= dd^c_\theta H$.
  Since this function depends continuously on $S$, by \cite[Theorem 2.4.4]{DS10JAG}, we can assume that $S$ is smooth.  So, if $U$ is a smooth potential of $S$,
   we can write
$H(\theta)=\left\langle U, {Q}_{\theta}\right\rangle$.
It is enough to estimate the mass of $dd^c H$ on $\mathbb{C}^{*}$. Indeed, the continuity of $H$ implies that $dd^c H$ has no mass on finite sets.
On $\C^*\times\widehat{X\times X}$ consider the smooth closed form
\[
T:=(dd^c\varphi_\theta+{\tau^*}(\gamma))\wedge{\tau^*}(\eta),
\]
where 
the $\varphi_\theta$'s are as after \eqref{eq:Spesific-theta} and
the $dd^c$ is taken in
$\C^*\times\widehat{X\times X}$,
and define the current
$\widehat{{Q}}:= T\wedge \tau^*(\Pi_1^*({ Q}))$.
It extends to a current on $\P^1\times\widehat{X\times X}$,
still   denoted by $\widehat{{Q}}$.
On this space, 
consider the two currents
$$
\widehat{{ Q}}_{U}:=\widehat{{ Q}} \wedge \tau^{*} \Pi_{2}^{*}(U) \quad \text { and } \quad \widehat{{ Q}}_{S}:=\widehat{{ Q}} \wedge \tau^{*} \Pi_{2}^{*}(S).
$$
These currents are smooth on $\mathbb{C}^{*} \times \widehat{X \times X}$. A direct computation gives $H= \left(\tau_{0}\right)_{*}\big(\widehat{{Q}}_{U}\big)$ and $dd^c H=\left(\tau_{0}\right)_{*}\big(\widehat{{Q}}_{S}\big)$ on $\mathbb{C}^{*}$. So, it is enough to estimate the mass of $\widehat{{Q}}_{S}$ on $\mathbb{C}^{*} \times \widehat{X \times X}$.
Recall that $T$ 
is smooth and closed on $\mathbb C^*\times\widehat{X\times X}$. 
By \cite[Lemma 2.4.2]{DS10JAG}, it extends
to $\mathbb P^1\times\widehat{X\times X}$
as a difference of positive closed currents
of bounded mass.
By \cite[Theorem 2.4.4]{DS10JAG},
we can write $T=T_1-T_2$ and
$S=S_1-S_2$, where $T_i$ and $S_i$
are positive closed currents, approximable
by smooth positive closed forms, with
$
\|T_i\|\leq C$ and  $\|S_i\|\leq C\|S\|_*$.
Here and below, $C$ denotes a constant
independent of $Q$ and $S$.
Choose smooth positive closed forms
$T_{i,n}\to T_i$ and $S_{i,m}\to S_i$
in the sense of currents, with uniformly
bounded masses. Set
$T_n:=T_{1,n}-T_{2,n}$ and
$S_m:=S_{1,m}-S_{2,m}$ and 
define
\[
\Omega_{n,m}:=(\Pi_1)_*\tau_*
\bigl((T_{1,n}+T_{2,n})
\wedge\tau^*\Pi_2^*(S_{1,m}+S_{2,m})\bigr).
\]
Since $\tau$ and $\Pi_1$ are submersions,
$\Omega_{n,m}$ is a smooth positive closed
$(s-1,s-1)$-form on $X$.
Its mass can be estimated cohomologically:
\[
\begin{aligned}
\|\Omega_{n,m}\|
\lesssim
\|T_{1,n}+T_{2,n}\|
\|S_{1,m}+S_{2,m}\|
\lesssim \|S\|_*.
\end{aligned}
\]
{\color{black} By definition of the $\bigstar$-norm, 
for every ${ \kappa}>0$ 
 there exist positive 
 (not necessarily closed)
  currents ${Q}_1$ and ${Q}_2$ such that 
${Q}={Q}_1-{Q}_2$ and 
$\sup_{\Phi\in\mathscr{H}_{ {s}-1}}\langle {Q}_1+{ Q}_2,\Phi\rangle\leq\Vert{ Q}\Vert ^{\bigstar}+{\kappa}$.}
 Since $T_{i,n}$ and $S_{i,m}$ are smooth
and positive, all the products below are
well-defined. By positivity,
 we have 
\begin{align*}
\big\|T_n\wedge\tau^*\Pi_1^*Q
\wedge\tau^*\Pi_2^*S_m\big\|_*&=\color{black} \Vert (T_{1,n}-T_{2,n})\wedge\tau^*\Pi_1^*({Q}_1-{Q}_2)\wedge\tau^*\Pi_2^*(S_{1,m}-S_{2,m})\Vert_*\\
&\leq
\big \langle \tau^*\Pi_1^*(Q_1+Q_2), 
(T_{1,n}+T_{2,n})
\wedge\tau^*\Pi_2^*(S_{1,m}+S_{2,m})\big\rangle\\
& =
\big\langle Q_1+Q_2,\Omega_{n,m}\big\rangle\leq
\|\Omega_{n,m}\|
\sup_{\Phi\in\mathscr H_{s-1}}
\big\langle Q_1+Q_2,\Phi\big\rangle\\
&\lesssim
\|S\|_*(\|Q\|^{\bigstar}+\kappa).
\end{align*}
If $\Omega_{n,m}=0$, the last inequality
is immediate; otherwise, we use
 the fact that
$\Omega_{n,m}/\|\Omega_{n,m}\|
\in\mathscr H_{s-1}$.
We now pass to the limit on
$\mathbb C^*\times\widehat{X\times X}$.
For fixed $m$, since $Q$ and $S_m$
are smooth and $T_n\to T$, we have
\[
T_n\wedge\tau^*\Pi_1^*Q
\wedge\tau^*\Pi_2^*S_m
\to
T\wedge\tau^*\Pi_1^*Q
\wedge\tau^*\Pi_2^*S_m.
\]
Then, since $T$ is smooth on this
open set and $S_m\to S$, the latter
currents converge to $\widehat Q_S$.
 By lower semicontinuity of the mass, $\|\widehat Q_S\|_{\mathbb C^*
\times\widehat{X\times X}}
\lesssim
 \|S\|_*(\|Q\|^{\bigstar}+\kappa)$. 
Letting $\kappa\to0$ and using
$dd^cH=(\tau_0)_*\widehat Q_S$
on $\mathbb C^*$, we obtain
$\|dd^cH\|_*
\lesssim\|S\|_*\|Q\|^{\bigstar}$,  which is the desired inequality.
{For a general $S$, the result follows
by approximation in the $*$-topology,
using \cite[Theorem 2.4.4]{DS10JAG}
and the lower semicontinuity of the
total variation
 of measures
under weak convergence.}
This completes the proof.
\end{proof}

Corollary \ref{cor:density-resp-bigstar-norm} gives smooth approximation in the
$\bigstar$-topology, without controlling the $\bigstar$-norms
of the approximates in terms of the $\bigstar$-norm of the
original current. The following proposition provides such a controlled approximation.
\begin{proposition}\label{prop:smooth-bigstar}
There exists a constant $c>0$, depending only on $X$, 
such that every $Q\in\mathscr{D}_{k-s+1}$ can be approximated 
in the $\bigstar$-topology by a sequence of smooth forms $Q_n\in\mathscr{D}_{k-s+1}$ satisfying $\{Q_n\}=\{Q\} $ and $\|Q_n\|^\bigstar
\leq c\|Q\|^\bigstar$ for every integer $n$. In particular, when $Q$ has a continuous super-potential
with respect to the $*$-topology,
$Q_n$ converges $SP^*$-uniformly to $Q$.   
\end{proposition}

\begin{proof}
If $Q=0$, we simply take $Q_n=0$, and the conclusion is immediate.
Assume that $Q\neq 0$.
Let $\Theta$ be a smooth positive closed $(k,k)$-form on
$\widehat{X\times X}$. For $\Phi\in\mathscr{H}_{s-1}$, set
\[
\mathcal{L}_\Theta(Q)
:=(\Pi_2)_*(\Pi_1^*Q\wedge\Theta)
\qquad
\text{and}
\qquad
\Omega:=(\Pi_1)_*(\Theta\wedge\Pi_2^*\Phi).
\]
Since
$\Pi_1$ is a submersion, the form $\Omega$ is smooth, positive, and closed. Computing its mass cohomologically,
we obtain $\|\Omega\|\leq c\|\Theta\|\|\Phi\|
\leq c\|\Theta\|$,
where $c>0$ is independent of $\Theta$ and $\Phi$.
Hence, $\Omega/(c\|\Theta\|)$ belongs to
$\mathscr{H}_{s-1}$.
For any positive decomposition $Q=Q^+-Q^-$, using the
definition of $\mathcal{L}_\Theta$, we obtain
\begin{align*}
\sup_{\Phi\in\mathscr{H}_{s-1}}
\langle\mathcal{L}_\Theta(Q^+)+\mathcal{L}_\Theta(Q^-),\Phi\rangle
&=\sup_{\Phi\in\mathscr{H}_{s-1}}\langle Q^++Q^-,\Omega\rangle\le c\|\Theta\|
\sup_{\Phi\in\mathscr{H}_{s-1}}\langle Q^++Q^-,\Phi\rangle.
\end{align*}
Taking the infimum over all such decompositions yields
\begin{equation}\label{eq:transform}
\|\mathcal{L}_\Theta(Q)\|^\bigstar
\leq c\|\Theta\|\|Q\|^\bigstar,
\end{equation}
where $c>0$ is independent of $Q$ and $\Theta$.
We now consider the specific form $\Theta=\Theta_\theta$
defined in (\ref{eq:Spesific-theta}).
Recall that $
\Theta_\theta=\Theta_\theta^+-\Theta_\theta^-$,
where $\Theta_\theta^\pm$ are smooth positive closed forms
whose masses are uniformly bounded for $\theta\neq 0$;
see \cite[\S 2.4]{DS10JAG}.
We denote by $\mathcal{L}_\theta$ and
$\mathcal{L}_\theta^\pm$ the transforms associated with
$\Theta_\theta$ and $\Theta_\theta^\pm$, respectively.
By (\ref{eq:transform}) and the triangle inequality,
there exists a constant $B\geq 1$, independent of
$\theta$ and $Q$, such that
\begin{equation}\label{eq:uniform}
\|\mathcal{L}_\theta(Q)\|^\bigstar
=
\|\mathcal{L}_\theta^+(Q)-\mathcal{L}_\theta^-(Q)\|^\bigstar
\leq B\|Q\|^\bigstar.
\end{equation}
Set $m:=k+2$. For each $1\leq j\leq m$,
choose a sequence of nonzero parameters
$\theta_{n,j}\in\mathbb{P}^1$ converging to $0$, and define
\[
A_n:=
\mathcal{L}_{\theta_{n,m}}\circ
\mathcal{L}_{\theta_{n,m-1}}\circ\cdots\circ
\mathcal{L}_{\theta_{n,1}}(Q).
\]
By 
\cite[Lemma 2.4.3]{DS10JAG}, we have
$A_n\to Q$ in the sense of currents.
Moreover, applying Proposition
\ref{prop:property-transform}(3) recursively, we obtain
$\{A_n\}=\{Q\}$ for every $n$.
By \cite[Corollary 2.3.3]{DS10JAG}, we also have
$A_n\in\mathcal{C}^1$.
Iterating (\ref{eq:uniform}) for
$\theta_{n,j}$, $1\leq j\leq m$, gives
\begin{equation}\label{eq:A_n-estimate-with-R}
\|A_n\|^\bigstar\leq B^m\|Q\|^\bigstar.
\end{equation}
By the same construction as in
Lemma \ref{lem:decom_S=alpha+S'} (1),  
we choose
a standard smooth closed representative $\widetilde{Q}$
of the cohomology class $\{Q\}$.
Since $\{A_n\}=\{Q\}$, we can write $A_n=\widetilde{Q}+dd^c U_n$ where $ U_n\in\mathcal{C}^2$.
We then approximate $U_n$ in the $\mathcal{C}^2$-topology
by smooth forms $V_n$ satisfying
$\|V_n-U_n\|_{\mathcal{C}^2}
\leq \frac{2^{-n}}{c_1}\|Q\|^\bigstar$ where $c_1>0$ is a constant to be specified below.
We claim that for every continuous real closed
$(k-s+1,k-s+1)$-form $E$, we have
\begin{equation}\label{eq:bigstar-norm-and-c^0-norm}
\|E\|^\bigstar\leq c_1\|E\|_{\infty},
\end{equation}
where $c_1>0$ is independent of $E$.
Indeed, since $E$ is continuous, there exists
a constant $c_1>0$, depending only on $X$, such that ${ |E|}\leq
c_1\|E\|_{_{\infty}}\omega^{k-s+1}$.
Consider the positive decomposition $E=E_1-E_2$, where
\[
E_1:=
\frac{E+c_1\|E\|_{\infty}\omega^{k-s+1}}{2},
\qquad
E_2:=
\frac{c_1\|E\|_{ \infty}\omega^{k-s+1}-E}{2}.
\]
By the definition of the $\bigstar$-norm, we obtain
\begin{align*}
\|E\|^\bigstar
&\leq
\sup_{\Phi\in\mathscr{H}_{s-1}}
\langle E_1+E_2,\Phi\rangle=
c_1\|E\|_{\infty}
\sup_{\Phi\in\mathscr{H}_{s-1}}
\langle\omega^{k-s+1},\Phi\rangle=
c_1\|E\|_{\infty}.
\end{align*}
This proves (\ref{eq:bigstar-norm-and-c^0-norm}). Now define the smooth closed forms $
Q_n:=\widetilde{Q}+dd^cV_n$.
Clearly, $\{Q_n\}=\{A_n\}=\{Q\}$.
Moreover,
$$\|Q_n-A_n\|_{\infty}=\|dd^c(V_n-U_n)\|_{\infty}\le \|V_n-U_n\|_{\mathcal{C}^2}\leq
2^{-n}\|Q\|^\bigstar/c_1.$$
Since $A_n\to Q$ in the sense of currents, it follows
that $Q_n\to Q$ in the sense of currents.
Finally, combining
(\ref{eq:A_n-estimate-with-R}) and
(\ref{eq:bigstar-norm-and-c^0-norm}), we obtain
\begin{align*}
\|Q_n\|^\bigstar
\leq
\|Q_n-A_n\|^\bigstar+\|A_n\|^\bigstar\leq
c_1\|Q_n-A_n\|_{\infty}
+\|A_n\|^\bigstar\leq
(B^m+1)\|Q\|^\bigstar.
\end{align*}
Denoting  $B^m+1$  by $c$ completes the
 proof of the first assertion.

We now prove the second assertion. By assumption $Q$ has a continuous super-potential $\mathscr{U}_Q$ in the $*$-topology. Since \(\theta_{n,j}\to 0\), it follows from \cite[Proposition 3.2.8]{DS10JAG} that \(A_n\) converges   $SP^*$-uniformly to $Q$;
indeed, the $SP^*$-uniform convergence is preserved under the application of a further structural transform, using the duality of super-potentials and the uniform $*$-norm bound in Proposition
\ref{prop:property-transform}(5).
Since the number of compositions is fixed, the conclusion follows by a telescoping argument. 
Equivalently, for $T\in
\mathscr D_s^0\cap\mathscr D_{s,*}^1$,
we have
$$
\sup_{\|T\|_*\leq 1}|\mathscr{U}_{A_n}(T)-\mathscr{U}_Q(T)|\to 0.
$$
We show that the approximating currents \(Q_n\) constructed in the first part
 of the proof
also converge \(SP^*\)-uniformly to \(Q\). 
Recall that $A_n=\widetilde{Q}+dd^c U_n$,  $Q_n=\widetilde{Q}+dd^c V_n$ and  $\|V_n-U_n\|_{\mathcal{C}^2}\leq \|Q\|^\bigstar/2^nc_1$.  The triangle inequality yields
\begin{align*}
    \sup_{\|T\|_*\leq 1}|\mathscr{U}_{Q_n}(T)-\mathscr{U}_Q(T)|\leq  \sup_{\|T\|_*\leq 1}|\mathscr{U}_{Q_n}(T)-\mathscr{U}_{A_n}(T)|+ \sup_{\|T\|_*\leq 1}|\mathscr{U}_{A_n}(T)-\mathscr{U}_Q(T)|.
\end{align*}
The second term tends to zero. For the first term, we have
\begin{align*}
    \sup_{\|T\|_*\leq 1}|\mathscr{U}_{Q_n}(T)-\mathscr{U}_{A_n}(T)|&= \sup_{\|T\|_*\leq 1}|\mathscr{U}_{dd^c(V_n-U_n)}(T)|=\sup_{\|T\|_*\leq 1}|\langle V_n-U_n, T\rangle|\\&
    \leq c_1\|V_n-U_n\|_{\infty}\sup_{\|T\|_*\leq 1}\|T\|_* \leq c_1\|V_n-U_n\|_{\mathcal{C}^2}\leq \|Q\|^\bigstar/2^n\to 0.
\end{align*}
Therefore, \(Q_n\) converges \(SP^*\)-uniformly to \(Q\). 
\end{proof}

\subsection{Green potential of \texorpdfstring{$dd^c$}{}-exact currents }
We briefly recall the construction of
a \emph{Green potential}
for $dd^c$-exact currents.
Let $\Delta\subset X\times X$ be the diagonal.
There exists a smooth real $(k,k)$-form $\alpha(x,y)$ cohomologous to $[\Delta]$ with $d_x\alpha=d_y\alpha=0$.
By \cite{BostGilletSoule,GilletSoule}, one can construct a $(k-1,k-1)$-form $K(x,y)$ on $X\times X$ such that
$dd^c K=[\Delta]-\alpha$
and satisfying
\begin{equation}\label{eq:K-log}
K(x,y)=O\big(|x-y|^{2-2k}\log|x-y|\big)
\end{equation}
near the diagonal.
For
 ${Q}\in\mathscr{D}_{k-{ s}+1}^0(X)$, we define
\[
U_{Q}(x):=\int_{y\in X}K(x,y)\wedge {Q}(y).
\]
 A direct computation shows that $dd^c U_{Q}={Q}$.
We refer to $U_{Q}$ as the 
\emph{Green potential}
of ${Q}$.
We also have the following result concerning the $L^{1}$-norm of
the Green potential
 $U_{Q}$. 
The version of this statement
with $\|{ Q}\|_*$ instead of $\|{ Q}\|^\bigstar$ 
can be found in \cite[Proposition~2.1]{DS05} and \cite[Proposition 2.4.7]{DS10JAG}. 
For
 convenience, we verify below that the estimates 
 remain
  valid  when the
   $\bigstar$-norm is used.

\begin{proposition}\label{prop:L_1-potential}
Take ${Q}\in\mathscr{D}_{k-{ s}+1}^0(X)$ and let
$U_{Q}$
be the
Green potential
 of ${Q}$.
Then there exists a constant
$c>0$, independent of ${Q}$, such that
 $$\Vert U_{Q}\Vert_{L^1}\leq c\Vert {Q}\Vert^{\bigstar}.$$  
\end{proposition}

\begin{proof} 
Let ${Q}={Q}_1-{Q}_2$ be a decomposition realizing $\|{Q}\|^{\bigstar}$ up to $\kappa$,
that is,
$\sup_{\Phi\in\mathscr{H}_{{ s}-1}}\langle {Q}_1+{ Q}_2,\Phi\rangle
\le
\|{Q}\|^{\bigstar}+\kappa$.
Since $\omega^{{s}-1}\in \mathscr{H}_{{s}-1}$,
this implies
$\|{Q}_1\|+\|{Q}_2\|
\le
\|{ Q}\|^{\bigstar}+\kappa$.

From 
(\ref{eq:K-log}), we have
$\|K(\cdot,y)\|_{L^1}\le c$
for some constant $c>0$ independent of $y$.
Writing $U_{Q}=U_1-U_2$ with
\[
U_i(x):=\int_{y\in X}K(x,y)\wedge {Q}_i(y),
\]
we obtain
$\|U_i\|_{L^1}\le c\|{Q}_i\|$
for
$i=1,2$.
Therefore,
\[
\|U_{Q}\|_{L^1}
\le
c(\|{Q}_1\|+\|{Q}_2\|)
\le
c (\|{Q}\|^{\bigstar}+\kappa).
\]
Letting ${\kappa}\to0$ gives the result.
\end{proof}

\end{document}